\documentclass[a4paper, 10pt]{amsart}
\usepackage{amsmath}
\usepackage{mathtools}

\usepackage{enumitem}

\usepackage{cite}

\newtheorem{thm}{Theorem}
\newtheorem{cor}[thm]{Corollary}
\newtheorem{lem}[thm]{Lemma}
\newtheorem{rmk}[thm]{Remark}

\title[Stationary states of an aggregation equation]{Stationary states of an aggregation equation with degenerate diffusion and bounded attractive potential}
\author{Gunnar Kaib}
\address{Applied Mathematics M\"unster\\University of M\"unster\\Einsteinstr.\ 62, D-48149 M\"unster, Germany}
\email{gunnar.kaib@uni-muenster.de}
\begin{document}
\bibliographystyle{plain}

\begin{abstract}
We investigate stationary solutions of a non-local aggregation equation with degenerate power-law diffusion and bounded attractive potential in arbitrary dimensions. Compact stationary solutions are characterized and compactness considerations are used to derive the existence of global minimizers of the corresponding energy depending on the prefactor of the degenerate diffusion for all exponents of the degenerate diffusion greater than one. We show that a global minimizer is compactly supported and, in case of quadratic diffusion, we prove that it is the unique stationary solution up to a translation. The existence of stationary solutions being only local minimizers is discussed.
\end{abstract}

\maketitle

\section{Introduction}
In this paper, we investigate stationary solutions of the non-local aggregation equation 
\begin{equation}
\label{eq1}
\partial_t\rho
=\nabla\cdot(\rho\nabla(\varepsilon\rho^{m-1}-G\ast\rho))
\end{equation}
in $\mathbb{R}^N$ where $G$ is a bounded purely attractive and integrable interaction potential and $\rho\in L^1(\mathbb{R}^N)\cap L^m(\mathbb{R}^N)$ a non-negative function satisfying $\|\rho\|_{L^1}=1$. Since we can rescale time in \eqref{eq1}, we assume the potential $G$ to be normalized, i.e. $\|G\|_{L^1}=1$, and state all results depending on the coefficient $\varepsilon>0$.

Due to the theory of gradient flows in \cite{ambrosiogiglisavare}, we can consider the aggregation equation \eqref{eq1} for probability density functions $\rho\in L^m(\mathbb{R}^N)$ formally as a gradient flow in the Wasserstein metric of the energy functional
\begin{equation}
\label{eq2}
E[\rho]=\int_{\mathbb{R}^N}\frac{\varepsilon}{m}\rho^m(x)\,dx-\frac{1}{2}\int_{\mathbb{R}^N}\int_{\mathbb{R}^N}G(x-y)\rho(y)\rho(x)\,dydx.
\end{equation}

Equation \eqref{eq1} is discussed in \cite{topazbertozzilewis} for $m=3$ to model biological aggregation. The non-linear diffusion models an anti-crowding motion acting locally repulsive while the attractive potential models a non-local aggregative behaviour. Performing some weakly non-linear and asymptotic analysis, so-called clump solutions with sharp edges are obtained in \cite{topazbertozzilewis} as stationary solutions of \eqref{eq1} which is one reasonable behaviour for biological swarming. 

In \cite{burgerdifrancescofranek}, it is shown for $N=1$ and $m=2$ that a unique stationary solution of \eqref{eq1} exists if the coefficient $\varepsilon$ satisfies $0<\varepsilon<1$ and that a stationary solution of \eqref{eq1} is compactly supported. In \cite{burgerfetecauhuang}, it can be found for $N=1$ a detailed discussion about stationary states of \eqref{eq1} for any $m>1$ partly based on numerical experiments. These observations are the starting point for this work.

Aggregation equations with a Newtonian or Bessel potential, which do not fall into the class of potentials considered in this paper, are widely studied due to the Keller-Segel model for chemotaxis \cite{kellersegel}. This model reads in a simplified form as
\begin{align}
\begin{split}
\label{kellersegel}
\partial_t\rho & =\Delta\rho-\nabla\cdot(\rho\nabla\phi),\\
-\Delta\phi+\alpha\phi & =\rho
\end{split}
\end{align}
with $\alpha\geq0$. For an extensive summary about results of the Keller-Segel system until 2003 see e.g.\ the review \cite{horstmann} and the references therein. Here, we only focus on thresholds which allow us to characterize properties of solutions of system \eqref{kellersegel} in $\mathbb{R}^N$.

In \cite{dolbeaultperthame,blanchetdolbeaultperthame}, it is shown for $N=2$ and $\alpha=0$, depending on the conserved total mass $M=\|\rho\|_{L^1}$, that for $M>8\pi$ solutions of system \eqref{kellersegel} blow up in finite time and that for $M<8\pi$ a global in time weak solution of \eqref{kellersegel} exists which converges with exponential rate \cite{camposdolbeault} to the self-similar solution of \eqref{kellersegel}. The critical case $M=8\pi$ is considered in \cite{blanchetcarrillomasmoudi} where the existence of a global in time solution with a finite second moment blowing up in infinite time is proved whereas the existence of a stationary solution with an infinite second moment is shown in the critical case in \cite{blanchetcarlencarrillo}.

Considering the Keller-Segel system \eqref{kellersegel} with a homogeneous non-linear diffusion
\begin{align}
\begin{split}
\label{kellersegeldegenerate}
\partial_t\rho & =\Delta\rho^m-\nabla\cdot(\rho\nabla\phi),\\
-\Delta\phi+\alpha\phi & =\rho,
\end{split}
\end{align}
we basically end up with equation \eqref{eq1} with a Newtonian or Bessel potential since we can rewrite \eqref{kellersegeldegenerate} as
\begin{equation}
\label{kellersegeldegeneraterewritten}
\partial_t\rho=\nabla\cdot(\nabla\rho^m-\rho\nabla(G\ast\rho))
\end{equation}
with a Newtonian or Bessel potential $G$. In \cite{bedrossianrodriguezbertozzi}, existence and uniqueness of weak solutions of \eqref{kellersegeldegeneraterewritten} are proved in $\mathbb{R}^N$ for $N\geq3$ for a wide class of attractive interaction potentials $G$ including the Newtonian and Bessel potentials. This result covers the well-posedness of weak solutions obtained in \cite{bertozzislepcev} for smooth attractive potentials and extends the existence theory for the Bessel potential in \cite{sugiyama2}. Existence and uniqueness of solutions to \eqref{kellersegeldegeneraterewritten} with respect to entropy solution are considered in \cite{burgercapassomorale} and for $N=1$ the uniqueness of solutions to \eqref{kellersegeldegeneraterewritten} is shown in \cite{burgerdifrancesco} using the pseudo-inverse of the Wasserstein distance.

Moreover, in \cite{bedrossianrodriguezbertozzi} it is determined for $N\geq3$ a critical exponent of the degenerate diffusion and a critical mass such that similar results as for the Keller-Segel system \eqref{kellersegel} for $N=2$ are obtained. In addition, for subcritical exponents it is shown that a global in time solution exists whereas in the supercritical case a finite time blow-up can occur for a certain class of problems. In particular, for the Newtonian and Bessel potential these results are also partly obtained in \cite{blanchetcarrillolaurencot} and \cite{sugiyama1,sugiyama2} respectively where the critical exponent is $m=2-\frac{2}{N}$. Some of the results in \cite{bedrossianrodriguezbertozzi} are extended in \cite{bedrossianrodriguez} to $\mathbb{R}^2$.

Regarding stationary solutions, equation \eqref{kellersegeldegeneraterewritten} is considered with a Newtonian or regularized Newtonian potential in \cite{kimyao}. Among others, it is shown via some mass comparison that there is a unique radially symmetric stationary solution for $m>2-\frac{2}{N}$ in dimensions $N\geq3$ which is monotonically decreasing and compactly supported. See also \cite{liebyau} for radial uniqueness results for stationary solutions of \eqref{kellersegeldegeneraterewritten} with a Newtonian potential. In addition, it is derived in \cite{kimyao} that continuous, radially symmetric and compactly supported solutions of equation \eqref{kellersegeldegeneraterewritten} with a Newtonian or regularized Newtonian potential are asymptotically converging in the subcritical case to the unique radial stationary solution as $t\rightarrow\infty$. The asymptotic behaviour towards stationary solutions in the critical case is considered in \cite{yao}.

The symmetry of stationary solutions of equation \eqref{kellersegeldegeneraterewritten} with a Newtonian potential is investigated in \cite{stroehmer} in three dimensions. It is shown via some variant of the moving plane method that compact stationary solutions are radially symmetric.

Combining the results in \cite{kimyao} and \cite{stroehmer}, one can conclude that compact stationary solutions of equation \eqref{kellersegeldegeneraterewritten} with a Newtonian interaction are unique in the subcritical case in three dimensions. The ideas of \cite{kimyao} and \cite{stroehmer} were recently used in \cite{carrillocastorinavolzone} to show that there is a unique compact stationary solution of equation \eqref{kellersegeldegeneraterewritten} with a Newtonian potential for $N=2$. This unique compact stationary solution coincides with the global minimizer of the corresponding energy. For results in the supercritical case $0<m<2-\frac{2}{N}$ and $N\geq3$ see e.g.\ \cite{bianliu}.

Recently, using continuous Steiner symmetrization techniques it was shown in \cite{carrillohittmeirvolzoneyao} for $N\geq1$ that all stationary solutions of equation \eqref{kellersegeldegeneraterewritten} are radially symmetric and monotonically decreasing if the interaction potential $G$ is radially symmetric, purely attractive and satisfies some growth assumptions. In particular, in case of a Newtonian interaction it is proved that there is a unique stationary solution. Moreover, for $N=2$ it is shown that a weak solution is converging towards this unique stationary solution as $t\rightarrow\infty$.

Furthermore, in \cite{chayeskimyao} the non-local aggregation equation \eqref{kellersegeldegeneraterewritten} with a symmetric attractive $C^2$-potential is considered on a torus. For $m=2$, it is shown in the subcritical case $\varepsilon>\max_{k\neq0}\big\{|\hat{G}(k)|\,\big|\,\hat{G}(k)>0\big\}$, where $\hat{G}$ denotes the Fourier transform of $G$, that the global minimizer is the constant solution and the energy of a weak solution converges exponentially fast to the energy of the constant solution with a rate depending on the initial data. In the supercritical case, it is found that a strictly positive solution, in particular the constant solution, is not a local minimizer of the energy. Similar results regarding such a phase transition were obtained before in \cite{chayespanferov} for an aggregation equation with linear diffusion.

Recently, many results were achieved considering non-local aggregation equations without diffusion in the space of probability measures. Instead of the diffusion and a purely attractive potential, a potential which acts locally repulsive and non-locally attractive is considered. Particularly, power-law potentials are widely studied. Depending on the potential, a rich variety of pattern formation is observed, see e.g.\ \cite{bertozzikolokolnikovsunuminskyvonbrecht}. In \cite{balaguecarrillolaurentraoul1}, the dimensionality of local minimizers of the corresponding energy is determined out of the strength of the local repulsive behaviour and in \cite{balaguecarrillolaurentraoul2}, the stability of radially symmetric solutions is investigated. Aggregation equations of this type do allow for stationary solutions not being radially symmetric, in contrast to aggregation equations with a degenerate diffusion and a suitable attractive potential as we consider here. This is due to the symmetry result in \cite{carrillohittmeirvolzoneyao}.

Regarding power-law potentials, the regularity of local minimizers is studied in \cite{carrillodelgadinomellet} and, if the repulsive part behaves like the Newtonian potential, the existence of unique radially symmetric and compactly supported stationary solutions is derived in \cite{fetecauhuangkolokolnikov,fetecauhuang}. Recently, for certain parameters of the power-law potential, explicit radially symmetric stationary solutions are obtained in \cite{carrillohuang}. 

Especially for power-law potentials, the existence of a global minimizer of the corresponding energy is shown in \cite{choksifetecautopaloglu} via the concentration-compactness principle of Lions \cite{lions} and in \cite{carrillochipothuang} via the discrete setting and some compactness observation. In \cite{canizocarrillopatacchini,simioneslepcevtopaloglu}, the existence of global minimizers is obtained for more general potentials. Additionally, it is shown in \cite{canizocarrillopatacchini} that global minimizers are compactly supported. The strategy in \cite{canizocarrillopatacchini} is to consider the problem first in a given ball instead of the whole space because this setting allows to derive the existence of a global minimizer. If there exists a probability measure such that its energy is less than the limit of the interaction potential at infinity, then it is shown that the diameter of the support of the global minimizer for the problem restricted to the ball is bounded from above independently of the size of the ball. This result is used to conclude that a compact global minimizer of the energy exists in the whole space. The approach in \cite{canizocarrillopatacchini} is similar to the one in \cite{auchmutybeals}.

In this paper, assuming that the interaction potential $G$ is bounded and radially strictly monotonically decreasing, we derive a sharp condition for stationary solutions of \eqref{eq1} being compactly supported which we then use to show that global minimizers of the energy functional \eqref{eq2}, if they exist, have compact support.

In \cite{bedrossian}, it is proved that a global minimizer of the energy functional \eqref{eq2} exists for $m=2$ and $0<\varepsilon<1$ as well as for $m>2$ and all $\varepsilon>0$. In \cite{burgerdifrancescofranek}, it is shown for $m=2$ that the threshold for the coefficient $\varepsilon$ is sharp. For a bounded, radially symmetric and purely attractive potential, we also characterize for $1<m<2$ a threshold $\varepsilon_0$ such that a global minimizer exists for all coefficients $0<\varepsilon\leq\varepsilon_0$. We prove this statement by adapting the compactness considerations of \cite{canizocarrillopatacchini} to our problem instead of following the approach in \cite{bedrossian} via the concentration-compactness principle of Lions \cite{lions} which only applies to $m\geq2$.

Moreover, for $m=2$ and $0<\varepsilon<1$ we show in arbitrary dimensions that the global minimizer of the energy functional \eqref{eq2} is unique up to a translation and that it is the unique stationary solution of \eqref{eq1} up to a translation. This stationary solution is radially symmetric and monotonically decreasing. We use several ideas from the proof for $N=1$ in \cite{burgerdifrancescofranek} as well as the symmetry result in \cite{carrillohittmeirvolzoneyao} and compactness considerations from \cite{canizocarrillopatacchini} to prove the uniqueness in higher dimensions.

For $m>2$, we prove that a compact stationary solution of \eqref{eq1} exists for all $\varepsilon>0$ if the interaction potential $G$ is bounded and strictly radially monotonically decreasing. For $1<m<2$, we show the existence of a constant $\varepsilon_0>0$ depending on the interaction potential and the exponent $m$ of the degenerate diffusion such that a compact stationary solution of \eqref{eq1} exists for all $0<\varepsilon\leq\varepsilon_0$. The results for $m\neq2$ are conjectured in \cite{burgerfetecauhuang}. Furthermore, our results complement some findings in \cite{burgerfetecauhuang} where it is shown for $N=1$ that for all $L>0$ there is some coefficient $\varepsilon>0$ such that a radially symmetric and monotonically decreasing stationary solution of \eqref{eq1} with support $[-L,L]$ exists. We extend this statement to arbitrary dimensions.

In addition, under the assumptions that for $m\neq2$ there is a unique radially symmetric and monotonically decreasing stationary solution of \eqref{eq1} with the coefficient $\varepsilon$ strictly increasing with the size of the support (which we prove for $m=2$ and which is indicated by numerical results in \cite{burgerfetecauhuang} for $m\neq2$), we show the following. For $1<m<2$, this unique stationary solution coincides with the global minimizer of the energy \eqref{eq2} for coefficients $\varepsilon>0$ not greater than $\varepsilon_0$ but loses this property for coefficients with a larger value. In particular, under these assumptions we show that $\varepsilon_0<\varepsilon_1$ where $\varepsilon_1$ denotes the threshold for the existence of a compact stationary solution of \eqref{eq1}. Numerically, it is also observed in \cite{burgerfetecauhuang} that stationary solutions of \eqref{eq1} being only local minimizers may exist.

Due to the additional non-linearity arising for $m\neq2$, it seems to be rather difficult to prove the uniqueness result which holds for $m=2$. However, numerical observations in \cite{burgerfetecauhuang} indicate that uniqueness should also be valid for $m\neq2$. So, this question stays an open problem for now. Besides, another interesting open problem not considered in the present paper concerns the long-time asymptotic of solutions to the aggregation equation \eqref{eq1} with a bounded attractive potential.

The paper is organized as follows: we state some assumptions of the interaction potential and the main result of the paper in Section \ref{notationandmainresult}. In Section \ref{preliminaries}, we show that under these assumptions stationary solutions of \eqref{eq1} are continuous in $\mathbb{R}^N$ as well as radially symmetric and monotonically decreasing. In Section \ref{compactness}, we derive a condition to characterize compact stationary solutions of \eqref{eq1} and prove that global minimizers of the energy functional \eqref{eq2} are compactly supported. The existence of a global minimizer of the energy given in \eqref{eq2} is shown in Section \ref{globalminimizers}. In Section \ref{stationarysolutions} we prove in arbitrary dimensions for $m=2$ uniqueness of stationary solutions of \eqref{eq1} up to a translation and consequently we show the uniqueness of the global minimizer of the corresponding energy. Finally, we discuss the existence of stationary solutions of \eqref{eq1} with positive energy in Section \ref{discussion}.

\section{Notation and main result}
\label{notationandmainresult}
In the entire paper, we suppose the interaction potential $G$ to satisfy that
\begin{enumerate}[label=(G\arabic*)]
\item\label{(G1)} $G$ is non-negative and $\mathrm{supp}\,G=\mathbb{R}^N$,
\item\label{(G2)} $G\in\,W^{1,1}(\mathbb{R}^N)\cap L^{\infty}(\mathbb{R}^N)\cap C^2(\mathbb{R}^N)$ with $\|G\|_{L^1}=1$,
\item\label{(G3)} $G$ is radially symmetric and strictly monotonically decreasing, i.e.\ $G(x)=g(|x|)$ and $g^\prime(r)<0$ for all $r>0$,
\item\label{(G4)} $g^{\prime\prime}(0)<0$ and $\lim_{r\rightarrow +\infty}g(r)=0$.
\end{enumerate}
Especially, $G$ being bounded plays an important role in our further considerations.

We define
\begin{equation}
\mathcal{P}^{M}(\mathbb{R}^N)\coloneqq\Big\{f\in L_+^1(\mathbb{R}^N)\,\Big|\,\int_{\mathbb{R}^N}f(x)\,dx=M\Big\}
\end{equation}
and for $M=1$ we use the abbreviated notation $\mathcal{P}(\mathbb{R}^N)$. Moreover, we write
\begin{equation}
\mathcal{P}_R(\mathbb{R}^N)\coloneqq\Big\{f\in L_+^1(\mathbb{R}^N)\,\Big|\,\int_{\mathbb{R}^N}f(x)\,dx=1,\,\mathrm{supp}\,f\subset\overline{B_R(0)}\Big\}.
\end{equation}
We consider stationary solutions of
\begin{equation}
\label{eq1b}
\partial_t\rho
=\nabla\cdot(\rho\nabla(\varepsilon\rho^{m-1}-G\ast\rho))
\end{equation}
in $\mathbb{R}^N$ and minimizers of the corresponding energy
\begin{equation}
\label{eq2b}
E[\rho]=\int_{\mathbb{R}^N}\frac{\varepsilon}{m}\rho^m(x)\,dx-\frac{1}{2}\int_{\mathbb{R}^N}\int_{\mathbb{R}^N}G(x-y)\rho(y)\rho(x)\,dydx
\end{equation}
in $L^m(\mathbb{R}^N)\cap\mathcal{P}(\mathbb{R}^N)$.

As a main result of this paper, we will prove that for $m=2$ a stationary solution of \eqref{eq1b} is unique up to a translation:
\begin{thm}
\label{uniquenesstheoremformequals2}
Let $G$ satisfy \ref{(G1)}-\ref{(G4)}, $m=2$ and $0<\varepsilon<1$, then there exists up to a translation a unique stationary solution $\rho\in L^2(\mathbb{R}^N)\cap\mathcal{P}(\mathbb{R}^N)$ of \eqref{eq1b}. Moreover, $\rho$ has the following properties:
\begin{itemize}
\item $\rho\in C^2(\mathrm{supp}\,\rho)\cap C(\mathbb{R}^N)$.
\item $\mathrm{supp}\,\rho=\overline{B_R(0)}$ for some $R>0$ depending on $\varepsilon$.
\item $\rho$ is radially symmetric.
\item $\rho$ is monotonically decreasing with increasing radius.
\item $\rho$ is the unique (up to a translation) global minimizer of the energy $E$ in $L^2(\mathbb{R}^N)\cap\mathcal{P}(\mathbb{R}^N)$.
\end{itemize} 
\end{thm}

\section{Continuity and symmetry of stationary solutions}
\label{preliminaries}
First, we recall some results being obtained among others in \cite[Section 2 and Section 3]{burgerdifrancescofranek} for $m=2$ which can be straightforwardly extended to all $m>1$.
\begin{lem}
\label{resultsfromburgerdifrancescofranek}
Let $G$ satisfy \ref{(G1)}-\ref{(G4)}, then the following results are satisfied for $m>1$ and $\varepsilon>0$:
\begin{enumerate}[label=(\roman*)]
\item\label{results(i)} The total mass $\int_{\mathbb{R}^N}\rho(x,t)\,dx$ of a solution of \eqref{eq1b} is preserved.
\item\label{results(ii)} The center of mass $\int_{\mathbb{R}^N}x\rho(x,t)\,dx$ of a solution of \eqref{eq1b} is preserved.
\item\label{results(iii)} If $\rho\in L^m(\mathbb{R}^N)\cap\mathcal{P}(\mathbb{R}^N)$ is a stationary solution of \eqref{eq1b}, then it holds that
\begin{equation*}
\varepsilon\rho^{m-1}-G\ast\rho=C
\end{equation*}
almost everywhere on every connected component of $\mathrm{supp}\,\rho$ where the constant $C$ may be different for each connected component. Moreover, we have $\rho^{m-1}\in C^2(\mathrm{supp}\,\rho)$ and $|\nabla\rho^{m-1}|\leq C^{\prime}$ in $\mathrm{supp}\,\rho$.
\item\label{results(iv)} If $\rho$ is a minimizer of the energy functional $E$ given in \eqref{eq2b} in $L^m(\mathbb{R}^N)\cap\mathcal{P}(\mathbb{R}^N)$, then we have
\begin{equation*}
\rho\nabla(\varepsilon\rho^{m-1}-G\ast\rho)=0
\end{equation*}
almost everywhere in $\mathbb{R}^N$.
\item\label{results(v)} If $\rho\in L^m(\mathbb{R}^N)\cap\mathcal{P}(\mathbb{R}^N)$ is a solution to
\begin{equation*}
\rho\nabla(\varepsilon\rho^{m-1}-G\ast\rho)=0,
\end{equation*}
then $\rho$ is a stationary point of the energy $E$ given in \eqref{eq2b}.
\item\label{results(vi)} If $\rho\in L^m(\mathbb{R}^N)\cap\mathcal{P}(\mathbb{R}^N)$ is a connected stationary solution of \eqref{eq1b}, then we have
\begin{equation*}
\varepsilon\rho^{m-1}(x)=(G\ast\rho)(x)+2E[\rho]-\int_{\mathbb{R}^N}\varepsilon\Big(\frac{2}{m}-1\Big)\rho^m(y)\,dy
\end{equation*}
for all $x\in\,\mathrm{supp}\,\rho$.
\end{enumerate}
\end{lem}
In addition, in \cite[Corollary 2.3]{burgerdifrancescofranek} it is shown for $N=1$ that a stationary solution of \eqref{eq2b} is continuous in $\mathbb{R}$. It is easy to prove this result for higher dimensions.
\begin{lem}
\label{continuityofstationarysolution}
If $\rho\in L^m(\mathbb{R}^N)\cap\mathcal{P}(\mathbb{R}^N)$ is a stationary solution of \eqref{eq1b}, $m>1$, $\varepsilon>0$ and $G$ satisfies \ref{(G1)}-\ref{(G4)}, then $\rho$ is continuous in $\mathbb{R}^N$.
\end{lem}
\begin{proof}
By Lemma \ref{resultsfromburgerdifrancescofranek}\ref{results(iii)}, a stationary solution $\rho\in L^m(\mathbb{R}^N)\cap\mathcal{P}(\mathbb{R}^N)$ of \eqref{eq1b} satisfies after differentiation and multiplication by $\rho$
\begin{equation*}
\varepsilon\frac{m-1}{m}\nabla(\rho^m)=\rho\nabla(G\ast\rho).
\end{equation*}
Due to \ref{(G2)}, \ref{(G4)} and $\rho\in\mathcal{P}(\mathbb{R}^N)$, we have $|\nabla(G\ast\rho)(x)|\leq C^\prime$ such that we obtain $\|\nabla(\rho^m)\|_{L^p}\leq \varepsilon^{-1}\frac{m}{m-1}C^\prime\|\rho\|_{L^p}$ for some $p\geq1$. Since $\rho\in L^1(\mathbb{R}^N)\cap L^m(\mathbb{R}^N)$, it follows that $\rho^m\in W^{1,1}(\mathbb{R}^N)$ and by Sobolev embedding we conclude for $N>1$ that $\rho^m\in L^p(\mathbb{R}^N)$ for all $1\leq p\leq\frac{N}{N-1}$. Hence, $\rho^m\in W^{1,p}(\mathbb{R}^N)$ and repeating the argument for $1<p<N$ it follows $\rho^m\in W^{1,q}(\mathbb{R}^N)$ for all $1\leq q\leq\frac{Np}{N-p}$. By Sobolev embedding, we conclude that $\rho^m\in C(\mathbb{R}^N)$ if $q>N$. Otherwise, we repeat the argument for $p<q<N$ until $\rho^m\in W^{1,r}(\mathbb{R}^N)$ for some $r>N$.
\end{proof}
Using a continuous Steiner symmetrization, in \cite[Section 2]{carrillohittmeirvolzoneyao} it is shown under suitable assumptions of the interaction potential that certain stationary solutions of an aggregation equation with a degenerate diffusion are radially symmetric and monotonically decreasing up to a translation. Due to our assumptions of the interaction potential $G$, in particular its radial symmetry, strict monotonicity and boundedness, and due to Lemma \ref{resultsfromburgerdifrancescofranek}\ref{results(iii)}, the symmetry result \cite[Theorem 2.2]{carrillohittmeirvolzoneyao} applies such that we can state the following theorem.
\begin{thm}
\label{symmetryofstationarysolutions}
If $\rho\in L^m(\mathbb{R}^N)\cap\mathcal{P}(\mathbb{R}^N)$ is a stationary solution of \eqref{eq1b}, $m>1$, $\varepsilon>0$ and $G$ satisfies \ref{(G1)}-\ref{(G4)}, then $\rho$ is radially symmetric and monotonically decreasing up to a translation. In particular, $\rho$ has a connected support.
\end{thm}

\section{Compactness of stationary solutions}
\label{compactness}
In \cite{canizocarrillopatacchini}, the existence of a global minimizer of the energy functional
\begin{equation*}
F[\mu]=\frac{1}{2}\int_{\mathbb{R}^N}\int_{\mathbb{R}^N}W(x-y)\,d\mu(y)d\mu(x)
\end{equation*}
is shown with $\mu$ being a probability measure and $W$ an appropriate interaction potential. One key idea to obtain this result is first considering a global minimizer of $F$ in the space of probability measures with support in a given ball and to show that in a fixed neighbourhood of each point of the support of this global minimizer there is at least a certain amount of mass concentrated. To derive this statement, one has to assume that the given ball is large enough and that the potential $W$ is unstable, i.e.\ that a probability measure $\mu$ with $F[\mu]<\lim_{|x|\rightarrow\infty}W(x)$ exists.

Under a suitable condition for the energy, we perform a calculation similar to a part of the proof in \cite[Lemma 2.6]{canizocarrillopatacchini} to conclude that stationary solutions of \eqref{eq1b}, i.e.\ solutions $\rho\in L^m(\mathbb{R}^N)\cap\mathcal{P}(\mathbb{R}^N)$ to the equation
\begin{equation}
\rho\nabla(\varepsilon\rho^{m-1}-G\ast\rho)=0
\end{equation}
in $\mathbb{R}^N$ are compactly supported.

\begin{lem}
\label{condforcompsupp}
Let $G$ satisfy \ref{(G1)}-\ref{(G4)}. If $\rho\in L^m(\mathbb{R}^N)\cap\mathcal{P}(\mathbb{R}^N)$ is a stationary solution of \eqref{eq1b}, $m>1$, $\varepsilon>0$ and if there exists a constant $K<0$ such that
\begin{equation}
\label{condforcompsuppeq1}
2E[\rho]-\int_{\mathbb{R}^N}\varepsilon\Big(\frac{2}{m}-1\Big)\rho^m(y)\,dy\leq K,
\end{equation}
then $\rho$ is compactly supported.
\end{lem}
\begin{proof}
Let $A<0$ with $K<A$ and $-G(0)<A$. By our assumptions of $G$, we know that there is an $r>0$ with $-G(x)>A$ for all $|x|>r$. Since $\rho$ has connected support due to Theorem \ref{symmetryofstationarysolutions}, by Lemma \ref{resultsfromburgerdifrancescofranek}\ref{results(vi)} it follows that
\begin{equation*}
\begin{aligned}
K & \geq 2E[\rho]-\int_{\mathbb{R}^N}\varepsilon\Big(\frac{2}{m}-1\Big)\rho^m(y)\,dy\\
& =\varepsilon\rho^{m-1}(z)-\int_{\mathbb{R}^N}G(z-y)\rho(y)\,dy\\
& \geq-G(0)\int_{B_r(z)}\rho(y)\,dy+A\int_{\mathbb{R}^N\setminus B_r(z)}\rho(y)\,dy\\
& =(-G(0)-A)\int_{B_r(z)}\rho(y)\,dy+A
\end{aligned}
\end{equation*}
for $z\in\,\mathrm{supp}\,\rho$. Due to $-G(0)-A<0$ and $K-A<0$, we obtain
\begin{equation*}
\int_{B_r(z)}\rho(y)\,dy\geq\frac{A-K}{A+G(0)}\eqqcolon C>0.
\end{equation*}
The constant $C$ is independent from $z$ such that this bound holds for every $z\in\,\mathrm{supp}\,\rho$ and we conclude that $\rho$ has a compact support because otherwise we obtain a contradiction.
\end{proof}
The proof of Lemma \ref{condforcompsupp} can be obtained more directly using Theorem \ref{symmetryofstationarysolutions}, Lemma \ref{resultsfromburgerdifrancescofranek}\ref{results(iv)} and the continuity of a stationary solution of \eqref{eq1b} in $\mathbb{R}^N$ (cf. Lemma \ref{continuityofstationarysolution}). However, we will use in Section \ref{globalminimizers} the above argumentation to show the existence of global minimizers of $E$ in $L^m(\mathbb{R}^N)\cap\mathcal{P}(\mathbb{R}^N)$.

Using the continuity of a stationary solution of \eqref{eq1b} in $\mathbb{R}^N$, we can actually show that $\rho\in L^m(\mathbb{R}^N)\cap\mathcal{P}(\mathbb{R}^N)$ being a compact stationary solution of \eqref{eq1b} is equivalent to the existence of a constant $K<0$ such that
\begin{equation*}
2E[\rho]-\int_{\mathbb{R}^N}\varepsilon\Big(\frac{2}{m}-1\Big)\rho^m(y)\,dy\leq K
\end{equation*}
is satisfied.
\begin{lem}
\label{compsuppsatisfycond}
Let $G$ satisfy \ref{(G1)}-\ref{(G4)}. If $\rho\in L^m(\mathbb{R}^N)\cap\mathcal{P}(\mathbb{R}^N)$ is a compact stationary solution of \eqref{eq1b}, $m>1$ and $\varepsilon>0$, then there exists a constant $K<0$ such that
\begin{equation*}
2E[\rho]-\int_{\mathbb{R}^N}\varepsilon\Big(\frac{2}{m}-1\Big)\rho^m(y)\,dy\leq K
\end{equation*}
is satisfied.
\end{lem}
\begin{proof}
Due to the symmetry result of Theorem \ref{symmetryofstationarysolutions} and since $\rho$ has compact support, there exists an $R<\infty$ such that we can assume that $\mathrm{supp}\,\rho=\overline{B_R(0)}$ without loss of generality. Then, by Lemma \ref{resultsfromburgerdifrancescofranek}\ref{results(vi)} and by the continuity of $\rho$ in $\mathbb{R}^N$ it follows that
\begin{equation*}
2E[\rho]-\int_{\mathbb{R}^N}\varepsilon\Big(\frac{2}{m}-1\Big)\rho^m(y)\,dy=-(G\ast\rho)(R).
\end{equation*}
Now, choose $K=-(G\ast\rho)(R)$.
\end{proof}

\begin{rmk}
\label{differentcompcond}
Using the definition of the energy $E$, we can rewrite the compactness condition in Lemma \ref{condforcompsupp} for a stationary solution $\rho\in L^m(\mathbb{R}^N)\cap\mathcal{P}(\mathbb{R}^N)$ of \eqref{eq1b} as
\begin{equation}
\label{differentcompcondeq1}
\int_{\mathbb{R}^N}\varepsilon\rho^m(x)\,dx-\int_{\mathbb{R}^N}(G\ast\rho)(x)\rho(x)\,dx\leq K
\end{equation}
with a constant $K<0$. Using Lemma \ref{resultsfromburgerdifrancescofranek}\ref{results(vi)}, we can also rewrite the compactness condition as
\begin{equation}
\label{differentcompcondeq2}
\varepsilon\rho^{m-1}(x)-(G\ast\rho)(x)\leq K
\end{equation}
for all $x\in\,\mathrm{supp}\,\rho$.

Moreover, due to \ref{(G3)} and \ref{(G4)}, a stationary solution $\rho\in L^m(\mathbb{R}^N)\cap\mathcal{P}(\mathbb{R}^N)$ of \eqref{eq1b} with non-compact support has to satisfy equation \eqref{condforcompsuppeq1}, \eqref{differentcompcondeq1} or \eqref{differentcompcondeq2} with equality and $K=0$.
\end{rmk}

\begin{cor}
\label{propcompandnoncompsol}
Let $G$ satisfy \ref{(G1)}-\ref{(G4)}, $m>1$ and $\varepsilon>0$.

A compact stationary solution $\rho\in L^m(\mathbb{R}^N)\cap\mathcal{P}(\mathbb{R}^N)$ of \eqref{eq1b} satisfies (after a translation such that the center of mass is in zero)
\begin{equation}
\label{equationforcompactsolution}
\varepsilon\rho^{m-1}(x)=(G\ast\rho)(x)-(G\ast\rho)(R)
\end{equation}
for all $x\in\,\mathrm{supp}\,\rho$ where $\mathrm{supp}\,\rho=\overline{B_R(0)}$ for some $R<\infty$. Moreover, we have
\begin{equation*}
\|\rho\|_{L^{m}}^{m}=\frac{1}{\varepsilon}\Big(\int_{\mathbb{R}^N}(G\ast\rho)(x)\rho(x)\,dx-(G\ast\rho)(R)\Big)
\end{equation*}
and the energy of a compact stationary solution of \eqref{eq1b} can be written as
\begin{equation*}
\begin{aligned}
E[\rho] & =\int_{\mathbb{R}^N}\varepsilon\Big(\frac{1}{m}-\frac{1}{2}\Big)\rho^m(y)\,dy-\frac{1}{2}(G\ast\rho)(R)\\
& =\Big(\frac{1}{m}-\frac{1}{2}\Big)\int_{\mathbb{R}^N}(G\ast\rho)(x)\rho(x)\,dx-\frac{1}{m}(G\ast\rho)(R).
\end{aligned}
\end{equation*}

A non-compact stationary solution $\rho\in L^m(\mathbb{R}^N)\cap\mathcal{P}(\mathbb{R}^N)$ of \eqref{eq1b} satisfies
\begin{equation}
\label{equationforconnectednoncompactsolution}
\varepsilon\rho^{m-1}(x)=(G\ast\rho)(x)
\end{equation}
in $\mathbb{R}^N$. Moreover, we have
\begin{equation*}
\|\rho\|_{L^{m}}^{m}=\frac{1}{\varepsilon}\int_{\mathbb{R}^N}(G\ast\rho)(x)\rho(x)\,dx
\end{equation*}
and the energy of a non-compact stationary solution of \eqref{eq1b} can be written as
\begin{equation*}
\begin{aligned}
E[\rho] & =\int_{\mathbb{R}^N}\varepsilon\Big(\frac{1}{m}-\frac{1}{2}\Big)\rho^m(y)\,dy\\
& =\Big(\frac{1}{m}-\frac{1}{2}\Big)\int_{\mathbb{R}^N}(G\ast\rho)(x)\rho(x)\,dx.
\end{aligned}
\end{equation*}
\end{cor}
\begin{proof}
By Lemma \ref{continuityofstationarysolution} and Theorem \ref{symmetryofstationarysolutions}, a stationary solution of \eqref{eq1b} is continuous as well as radially symmetric and monotonically decreasing. Using Lemma \ref{resultsfromburgerdifrancescofranek}\ref{results(vi)}, we immediately obtain equations \eqref{equationforcompactsolution} and \eqref{equationforconnectednoncompactsolution}. These equation can be used to calculate the $L^m$-norm and the energy $E$ of $\rho$. 
\end{proof}

\begin{rmk}
\label{mequals2nononcompact}
For $m=2$, there cannot exist a non-compact stationary solution $\rho\in L^2(\mathbb{R}^N)\cap\mathcal{P}(\mathbb{R}^N)$ of \eqref{eq1b}, which satisfies $\mathrm{supp}\,\rho=\mathbb{R}^N$ due to Theorem \ref{symmetryofstationarysolutions}, since when integrating \eqref{equationforconnectednoncompactsolution} over $\mathbb{R}^N$ we obtain $\varepsilon=1$. Using a Fourier transform, it is shown in the proof of \cite[Theorem 3.5]{burgerdifrancescofranek} that if $\varepsilon=1$ (the critical case for $m=2$), then no $\rho\in L^2(\mathbb{R}^N)\cap\mathcal{P}(\mathbb{R}^N)$ satisfying $\varepsilon\rho(x)=(G\ast\rho)(x)$ in $\mathbb{R}^N$ exists. 

For $m\neq 2$, we cannot exclude the existence of a non-compact stationary solution $\rho\in L^m(\mathbb{R}^N)\cap\mathcal{P}(\mathbb{R}^N)$ of \eqref{eq1b} via the Fourier approach.
\end{rmk}

\begin{thm}
\label{negenergy}
Let $G$ satisfy \ref{(G1)}-\ref{(G4)}. Every stationary solution $\rho\in L^m(\mathbb{R}^N)\cap\mathcal{P}(\mathbb{R}^N)$ of \eqref{eq1b} (compact or non-compact) has negative energy if $m\geq 2$ and $\varepsilon>0$.
\end{thm}
\begin{proof}
By Corollary \ref{propcompandnoncompsol}, for $m\geq2$ no compact stationary solution $\rho\in L^m(\mathbb{R}^N)\cap\mathcal{P}(\mathbb{R}^N)$ of \eqref{eq1b} with non-negative energy exists. Furthermore, for $m>2$ no non-compact stationary solution $\rho\in L^m(\mathbb{R}^N)\cap\mathcal{P}(\mathbb{R}^N)$ of \eqref{eq1b} with non-negative energy exists due to Corollary \ref{propcompandnoncompsol} and the existence of a non-compact stationary solution of \eqref{eq1b} for $m=2$ is excluded by Remark \ref{mequals2nononcompact}. 
\end{proof}

\begin{rmk}
\label{stationarysolutionswithpositiveenergy}
By Corollary \ref{propcompandnoncompsol}, we directly see that for $1<m<2$ and $\varepsilon>0$ a non-compact stationary solution $\rho\in L^m(\mathbb{R}^N)\cap\mathcal{P}(\mathbb{R}^N)$ of \eqref{eq1b}, if it exists, has positive energy. Moreover, for $1<m<2$ and $\varepsilon>0$ the compactness condition from Lemma \ref{condforcompsupp} does not exclude the existence of compact stationary solutions of \eqref{eq1b} with positive energy in contrast to $m\geq 2$.
\end{rmk}

\begin{thm}
\label{globmincomp}
Let $G$ satisfy \ref{(G1)}-\ref{(G4)}, $m>1$ and $\varepsilon>0$. If a global minimizer of the energy $E$ in $L^m(\mathbb{R}^N)\cap\mathcal{P}(\mathbb{R}^N)$ exists, it is compactly supported.
\end{thm}
\begin{proof}
For $1<m\leq 2$, a global minimizer, if it exists (cf.\ Section \ref{globalminimizers}), has non-positive energy such that, due to Lemma \ref{resultsfromburgerdifrancescofranek}\ref{results(iv)} and \ref{results(vi)}, it applies Lemma \ref{condforcompsupp}.

Therefore, let $m>2$ and assume that $\rho$ is a non-compact global minimizer, i.e.\ 
\begin{equation*}
\varepsilon\rho^{m-1}(x)=(G\ast\rho)(x).
\end{equation*}
Since $\rho$ is radially symmetric and monotonically decreasing by Riesz symmetric decreasing rearrangement inequality and since we have $\int_{\mathbb{R}^N}\rho(x)\,dx=1$, it holds that $\rho(x)|B_{|x|}(0)|\leq1$ such that we conclude that $\rho(x)\leq C|x|^{-N}$. Moreover, using that $G$ is positive and radially symmetric by conditions \ref{(G1)} and \ref{(G3)}, for $|x|>1$ we estimate that
\begin{equation*}
\begin{aligned}
(G\ast\rho)(x)&=\int_{\mathbb{R}^N}G(x-y)\rho(y)\,dy\geq\rho(x)\int_{B_{|x|}(0)}G(x-y)\,dy\\
& \geq\rho(x)\int_{B_1(x-\frac{x}{|x|})}G(x-y)\,dy=\rho(x)\int_{B_1(0)}G\Big(\frac{x}{|x|}-y\Big)\,dy=C^\prime\rho(x)
\end{aligned}
\end{equation*}
with $C^\prime$ dependent on $G$ but independent from $x$ because of the radial symmetry of $G$. Therefore, we have
\begin{equation*}
\varepsilon\rho^{m-1}(x)-(G\ast\rho)(x)\leq\rho(x)(\varepsilon C^{m-2}|x|^{-N(m-2)}-C^\prime)
\end{equation*}
which contradicts for large $|x|>1$ the assumption that $\rho$ is a non-compact stationary solution of (1) since $m>2$.
\end{proof}

\section{Existence of global minimizers}
\label{globalminimizers}
In the following, we prove the existence of a global minimizer of the energy
\begin{equation*}
E[\rho]=\int_{\mathbb{R}^N}\frac{\varepsilon}{m}\rho^m(x)\,dx-\frac{1}{2}\int_{\mathbb{R}^N}\int_{\mathbb{R}^N}G(x-y)\rho(y)\rho(x)\,dydx
\end{equation*}
in $L^m(\mathbb{R}^N)\cap\mathcal{P}(\mathbb{R}^N)$ without using the concentration-compactness principle \cite{lions} if the interaction potential $G$ satisfies \ref{(G1)}-\ref{(G4)}. For $m\geq2$, we obtain the same result as in \cite{bedrossian} and for $1<m<2$ we determine coefficients $\varepsilon>0$ which allow the existence of a global minimizer.

Using the concentration-compactness principle it is proved in \cite[Theorem II.1]{lions} that a global minimizer of $E$ exists if and only if the strict subadditivity condition
\begin{equation}
\inf_{\rho\in L^m\cap\mathcal{P}}E[\rho]<\inf_{\rho\in L^m\cap\mathcal{P}^M}E[\rho]+\inf_{\rho\in L^m\cap\mathcal{P}^{(1-M)}}E[\rho]
\end{equation}
is satisfied for all $M\in(0,1)$.

Moreover, in \cite{lions} the interaction potential does not need to be necessarily radially symmetric and monotonically decreasing. However, verifying that the strict subadditivity condition is satisfied may not be straightforward. In \cite[Corollary II.1]{lions} this condition is shown to be satisfied if either $1<m\leq2$ and $\inf_{\rho\in L^m\cap\mathcal{P}}E[\rho]<0$ or $G(\tau x)\geq\tau^{-k}G(x)$ for all $\tau\geq1$ and almost every $x\in\mathbb{R}^N$ and $\inf_{\rho\in L^m\cap\mathcal{P}}E[\rho]<0$.

If the interaction potential $G$ is non-negative, radially symmetric and monotonically decreasing, it is proved in \cite{bedrossian} that a global minimizer of the energy $E$ in $L^m(\mathbb{R}^N)\cap\mathcal{P}(\mathbb{R}^N)$ exists for any coefficient $\varepsilon>0$ if $m>2$ and for any $0<\varepsilon<1$ if $m=2$. Moreover, it is shown that a global minimizer is radially symmetric and monotonically decreasing if the interaction potential $G$ is radially symmetric and strictly monotonically decreasing.

The strategy of the proof is to use a suitable scaling argument to observe that, under the above assumptions of $\varepsilon$ and $m$, for every $M>0$ a function $\rho\in C_0^\infty(\mathbb{R}^N)\cap L^m(\mathbb{R}^N)\cap\mathcal{P}^{M}(\mathbb{R}^N)$ with $E[\rho]<0$ exists. Together with Riesz symmetric decreasing rearrangement inequality, this statement is used to show that for $M_1>M_2$ the relation
\begin{equation}
\label{massminrelation}
\inf_{\rho\in L^m\cap\mathcal{P}^{M_1}}E[\rho]<\inf_{\rho\in L^m\cap\mathcal{P}^{M_2}}E[\rho].
\end{equation}
is satisfied. This conclusion is used to derive the existence of a radially symmetric and monotonically decreasing global minimizer using Riesz symmetric decreasing rearrangement inequality and following the approach of the concentration-compactness principle of Lions \cite{lions}. More precisely, the inequality \eqref{massminrelation} is needed to rule out that dichotomy occurs.

We prove the existence of global minimizers of $E$ with a potential satisfying \ref{(G1)}-\ref{(G4)} using some ideas in \cite{canizocarrillopatacchini} instead of the concentration-compactness principle \cite{lions}. In \cite{canizocarrillopatacchini}, the existence of global minimizers is investigated in a slightly different setting. More precisely, an interaction energy functional
\begin{equation*}
F[\mu]=\frac{1}{2}\int_{\mathbb{R}^N}\int_{\mathbb{R}^N}W(x-y)\,d\mu(y)d\mu(x)
\end{equation*}
with $\mu$ being a probability measure and $W$ an appropriate interaction potential is analysed. It is shown that if considering the problem on a given ball, then the diameter of the support of a global minimizer of $F$ is bounded independently of the size of the given ball. This observation is used to conclude that a global minimizer of the problem in the whole space exists. A similar approach can also be found in \cite{auchmutybeals}.

Moreover, this strategy allows to consider the case $\inf_{\rho\in L^m\cap\mathcal{P}}E[\rho]=0$ for $1<m<2$ which is not covered by the approach in \cite{lions}. We will obtain in this section the following result:
\begin{thm}
\label{existenceofglobalmin}
Let $G$ satisfy \ref{(G1)}-\ref{(G4)}. Let one of the following conditions be satisfied:
\begin{itemize}
\item[(i)] $1<m<2$ and $0<\varepsilon\leq\varepsilon_0\coloneqq\sup_{\rho\in L^m\cap\mathcal{P}}\frac{m}{2}\|\rho\|_{L^m}^{-m}\int_{\mathbb{R}^N}(G\ast\rho)(x)\rho(x)\,dx$
\item[(ii)] $m=2$ and $0<\varepsilon<1$
\item[(iii)] $m>2$ and $\varepsilon>0$
\end{itemize}
Then, there exists a global minimizer of the energy $E$ in $L^m(\mathbb{R}^N)\cap\mathcal{P}(\mathbb{R}^N)$ which is radially symmetric and monotonically decreasing.

If $1<m\leq2$ and the corresponding condition (i) or (ii) for the coefficient $\varepsilon$ is not satisfied, then there exists no global minimizer of the energy $E$ in $L^m(\mathbb{R}^N)\cap\mathcal{P}(\mathbb{R}^N)$.
\end{thm}
First, let us remark that a global minimizer of $E$ in $L^m(\mathbb{R}^N)\cap\mathcal{P}(\mathbb{R}^N)$ cannot have positive energy. So, if we want to obtain a global minimizer for $1<m<2$, we have to assume at least that $\varepsilon$ is chosen such that a function $\rho\in L^m(\mathbb{R}^N)\cap\mathcal{P}(\mathbb{R}^N)$ with $E[\rho]\leq0$ exists.

If we assume that a function $\rho\in L^m(\mathbb{R}^N)\cap\mathcal{P}(\mathbb{R}^N)$ with $E[\rho]=0$ exists but none with $E[\rho]<0$, then $\rho$ is already a global minimizer. Therefore, it is enough to consider the case that $E[\rho]<0$ for some $\rho\in L^m(\mathbb{R}^N)\cap\mathcal{P}(\mathbb{R}^N)$. Following an analogous idea given in \cite[Lemma 2.5]{canizocarrillopatacchini}, we define
\begin{equation*}
\rho_n\coloneqq\frac{1}{\int_{B_n(0)}\rho(x)\,dx}\chi_{B_n(0)}\rho.
\end{equation*}
By Lebesgue's monotone convergence theorem, we obtain
\begin{equation*}
E[\rho_n]\rightarrow E[\rho]\quad\mathrm{for}\quad n\rightarrow\infty
\end{equation*}
such that a function $\hat{\rho}\in L^m(\mathbb{R}^N)\cap\mathcal{P}_S(\mathbb{R}^N)$ with $E[\hat{\rho}]<0$ exists for some large enough but finite constant $S$. In particular, there exists a $\hat{\rho}\in L^m(\mathbb{R}^N)\cap\mathcal{P}_R(\mathbb{R}^N)$ with $E[\hat{\rho}]<0$ for all $R\geq S$. In the following, we always denote by $S$ the finite constant we just obtained.
\begin{lem}
\label{globminrestricteddomain}
Let $m>1$, $\varepsilon>0$ and let $G$ satisfy \ref{(G1)}-\ref{(G4)}. Then, there exists a global minimizer of $E$ in $L^m(\mathbb{R}^N)\cap\mathcal{P}_R(\mathbb{R}^N)$ for every $R>0$.
\end{lem}
\begin{proof}
Since $G$ is bounded from above, we have $\inf_{\rho\in L^m\cap\mathcal{P}_R}E[\rho]>-\infty$. Let $(\rho_n)_{n\in\mathbb{N}}$ denote a minimizing sequence for $E$ in $L^m(\mathbb{R}^N)\cap\mathcal{P}_R(\mathbb{R}^N)$. Since $G$ is bounded from above and $(\rho_n)_{n\in\mathbb{N}}$ a minimizing sequence, it holds that $(\rho_n)_{n\in\mathbb{N}}$ is bounded in $L^m(\mathbb{R}^N)$. So, there exists a subsequence $(\rho_{n_k})_{k\in\mathbb{N}}$ such that $\rho_{n_k}$ converges weakly in $L^m(\mathbb{R}^N)$ to some $\rho_R\in L^m(\mathbb{R}^N)$. Further, $L^m(\mathbb{R}^N)\cap\mathcal{P}_R(\mathbb{R}^N)$ is convex and closed in $L^m(\mathbb{R}^N)$ since $L^m(\overline{B_R(0)})$ is closed in $L^m(\mathbb{R}^N)$ and since the embedding of $L^m(\overline{B_R(0)})$ in $L^1(\overline{B_R(0)})$ as well as the mass constraint is continuous in $L^1(\overline{B_R(0)})$. Therefore, $L^m({\mathbb{R}^N})\cap\mathcal{P}_R(\mathbb{R}^N)$ is weakly closed in $L^m(\mathbb{R}^N)$ and we have $\rho_R\in L^m(\mathbb{R}^N)\cap\mathcal{P}_R(\mathbb{R}^N)$. Obviously, $E$ is weakly lower semi-continuous in $L^m(\mathbb{R}^N)$ such that $E[\rho_R]=\inf_{\rho\in L^m\cap\mathcal{P}_R}E[\rho]$.
\end{proof}
Performing similar calculations as in \cite[Lemma 2.3]{canizocarrillopatacchini}, we get the following result for a global minimizer of the energy in the restricted setting.
\begin{lem}
Let $m>1$, $\varepsilon>0$ and let $G$ satisfy \ref{(G1)}-\ref{(G4)}. Let $\rho_R$ be a global minimizer of $E$ in $L^m(\mathbb{R}^N)\cap\mathcal{P}_R(\mathbb{R}^N)$, then we have
\begin{equation}
\label{equationforglobalminimizer}
\varepsilon\rho_R^{m-1}(x)-(G\ast\rho_R)(x)=2E[\rho_R]-\int_{\mathbb{R}^N}\varepsilon\Big(\frac{2}{m}-1\Big)\rho_R^m(y)\,dy
\end{equation}
almost everywhere in $\mathrm{supp}\,\rho_R$.
\end{lem}
\begin{proof}
Let $\psi\in C_0^\infty(\mathbb{R}^N)$ and define
\begin{equation*}
f(x)\coloneqq\Big(\psi(x)-\int_{\mathbb{R}^N}\psi(y)\rho_R(y)\,dy\Big)\rho_R(x)
\end{equation*}
for all $x\in\mathbb{R}^N$ and $\rho_{\delta}\coloneqq\rho_R+\delta f$ with $\delta>0$. Then, it holds that
\begin{equation*}
\int_{\mathbb{R}^N}\rho_{\delta}(x)\,dx=\int_{\mathbb{R}^N}\rho_R(x)\,dx+\delta\int_{\mathbb{R}^N}f(x)\,dx=1.
\end{equation*}
Moreover, we have
\begin{equation*}
\psi(x)-\int_{\mathbb{R}^N}\psi(y)\rho_R(y)\,dy\geq -2\|\psi\|_{L^\infty(\mathbb{R}^N)}
\end{equation*}
such that $\rho_{\delta}(x)\geq(1-2\delta\|\psi\|_{L^{\infty}(\mathbb{R}^N)})\rho_{R}(x)$. So, for sufficiently small $\delta>0$ we have $\rho_{\delta}\in L^m(\mathbb{R}^N)\cap\mathcal{P}_R(\mathbb{R}^N)$ and since $\rho_R$ is a global minimizer in $L^m(\mathbb{R}^N)\cap\mathcal{P}_R(\mathbb{R}^N)$, it follows that $E[\rho_{\delta}]\geq E[\rho_R]$. Therefore, we obtain
\begin{equation*}
\begin{aligned}
& \frac{1}{\delta}\big(E[\rho_{\delta}]-E[\rho_R]\big)\big|_{\delta=0}\\
& \quad =\int_{\mathbb{R}^N}\Big(\varepsilon\rho_R^{m-1}(x)-(G\ast\rho_R)(x)-2E[\rho_R]\Big)\rho_R(x)\psi(x)\,dx\\
 & \quad\quad +\int_{\mathbb{R}^N}\Big(\int_{\mathbb{R}^N}\varepsilon\Big(\frac{2}{m}-1\Big)\rho_R^m(y)\,dy\Big)\rho_R(x)\psi(x)\,dx\\
& \quad \geq0.
\end{aligned}
\end{equation*}
Since this inequality holds for all $\psi\in C_0^\infty(\mathbb{R}^N)$, we conclude that 
\begin{equation*}
\varepsilon\rho_R^{m-1}(x)-(G\ast\rho_R)(x)=2E[\rho_R]-\int_{\mathbb{R}^N}\varepsilon\Big(\frac{2}{m}-1\Big)\rho_R^m(y)\,dy
\end{equation*}
holds almost everywhere in $\mathrm{supp}\,\rho_R$.
\end{proof}
Taking into consideration the proof of Lemma \ref{condforcompsupp} and the strategy in \cite{canizocarrillopatacchini}, for $R\geq S$ we obtain for each point in the support of a global minimizer $\rho_R$ a lower bound for the mass in a certain neighbourhood. The size of the neighbourhood as well as of the lower bound do not depend on $R$. This is because of $E[\rho_R]\leq E[\rho_S]<0$ for all $R\geq S$, where $\rho_R$ denotes a global minimizer of $E$ in $L^m(\mathbb{R}^N)\cap\mathcal{P}_R(\mathbb{R}^N)$ and $\rho_S$ in $L^m(\mathbb{R}^N)\cap\mathcal{P}_S(\mathbb{R}^N)$ respectively, since we can now choose
\begin{equation*}
K=2E[\rho_S]
\end{equation*}
in the proof of Lemma \ref{condforcompsupp}.

By Riesz symmetric decreasing rearrangement inequality, we know that the global minimizer $\rho_R$ is radially symmetric and monotonically decreasing, i.e.\ it has connected support. As in \cite[Lemma 2.9]{canizocarrillopatacchini}, due to our above considerations about mass concentration, we obtain an upper bound $D<+\infty$ for the diameter of the support of $\rho_R$ for all $R>0$. Using that, as in \cite[Lemma 2.10]{canizocarrillopatacchini} we prove the following statement.
\begin{lem}
\label{globalminimizerformsmaller2}
Let $G$ satisfy \ref{(G1)}-\ref{(G4)}. Let $1<m<2$ and $\varepsilon>0$ be small enough such that a function $\rho\in L^m(\mathbb{R}^N)\cap\mathcal{P}(\mathbb{R}^N)$ with $E[\rho]\leq0$ exists. Then, there exists a global minimizer of $E$ in $L^m(\mathbb{R}^N)\cap\mathcal{P}(\mathbb{R}^N)$.
\end{lem}
\begin{proof}
As argued before, we only need to consider the case $E[\rho]<0$. We can conclude analogously as in the proof for a global minimizer in \cite[Lemma 2.10]{canizocarrillopatacchini} and include the proof for the convenience of the reader. 

Let $D$ denote the upper bound of the diameter of the support of a global minimizer $\rho_R$ for all $R>0$ and let $\rho_D$ be a global minimizer of $E$ in $L^m(\mathbb{R}^N)\cap\mathcal{P}_D(\mathbb{R}^N)$. Then, we show that $\rho_D$ actually is a global minimizer of $E$ in $L^m(\mathbb{R}^N)\cap\mathcal{P}(\mathbb{R}^N)$.

First, consider a function $\rho\in L^m(\mathbb{R}^N)\cap\mathcal{P}(\mathbb{R}^N)$ with compact support, i.e.\ $\rho\in L^m(\mathbb{R}^N)\cap\mathcal{P}_R(\mathbb{R}^N)$ for some $R>0$. Then, $E[\rho_R]\leq E[\rho]$ with $\rho_R$ denoting a global minimizer of $E$ in $L^m(\mathbb{R}^N)\cap\mathcal{P}_R(\mathbb{R}^N)$. Since the diameter of the support of $\rho_R$ is bounded by $D$, we have $\rho_R\in L^m(\mathbb{R}^N)\cap\mathcal{P}_D(\mathbb{R}^N)$. Therefore, it holds that $E[\rho_D]\leq E[\rho]$ for all $\rho\in L^m(\mathbb{R}^N)\cap\mathcal{P}(\mathbb{R}^N)$ with compact support.

Now, consider some function $\rho\in L^m(\mathbb{R}^N)\cap\mathcal{P}(\mathbb{R}^N)$ without further restrictions. If $n$ is chosen large enough, it holds that $\rho(B_n(0))>0$. Define
\begin{equation*}
\rho_n\coloneqq\frac{1}{\int_{B_n(0)}\rho(x)\,dx}\chi_{B_n(0)}\rho
\end{equation*}
for $n$ large enough and observe that $E[\rho_n]\rightarrow E[\rho]$ as $n\rightarrow\infty$ by Lebesgue's monotone convergence theorem. Since $\rho_n$ is compactly supported, we have $E[\rho_D]\leq E[\rho_n]$ for all $n$. Therefore, it holds that $E[\rho_D]\leq E[\rho]$. Due to $E[\rho_D]<0$ by assumption, we have shown that $\rho_D$ is a global minimizer of $E$ in $L^m(\mathbb{R}^N)\cap\mathcal{P}(\mathbb{R}^N)$.
\end{proof}
By our previous approach, we can also prove the existence of global minimizers for $m\geq2$, which is obtained in \cite{bedrossian} via the concentration-compactness principle of Lions \cite{lions}, if we assume the interaction potential $G$ to be bounded.

Notice that for $m=2$, no function $\rho\in L^m(\mathbb{R}^N)\cap\mathcal{P}(\mathbb{R}^N)$ with $E[\rho]=0$ exists due to Lemma \ref{compsuppsatisfycond} and since no non-compact stationary solution of \eqref{eq1b} exists for $m=2$ (cf.\ Remark \ref{mequals2nononcompact}). So, for $m=2$ we also obtain as in \cite{bedrossian} the existence of a global minimizer for every $0<\varepsilon<1$. Furthermore, there cannot exist a global minimizer of $E$ in $L^2(\mathbb{R}^N)\cap\mathcal{P}(\mathbb{R}^N)$ for $\varepsilon\geq1$ because of the non-existence of a stationary solution of \eqref{eq1b} as shown in \cite[Lemma 3.4 and Theorem 3.5]{burgerdifrancescofranek}.

Keep in mind, it is also impossible for $1<m<2$ that a global minimizer exists for coefficients $\varepsilon>0$  not satisfying the condition in Lemma \ref{globalminimizerformsmaller2}.

For $m>2$, we assume by contradiction that there exists no upper bound for the diameter of the support of global minimizers $\rho_R$ of $E$ in $L^m(\mathbb{R}^N)\cap\mathcal{P}_R(\mathbb{R}^N)$ for all $R>0$. Then, we can use equation \eqref{equationforglobalminimizer} and the argumentation in the proof of Theorem \ref{globmincomp} to show that a global minimizer $\rho_S$ of $E$ in $L^m(\mathbb{R}^N)\cap\mathcal{P}_S(\mathbb{R}^N)$ satisfies
\begin{equation*}
2E[\rho_S]-\int_{\mathbb{R}^N}\varepsilon\Big(\frac{2}{m}-1\Big)\rho_S^m(y)\,dy\leq\rho_S(x)(\varepsilon C^{m-2}|x|^{-N(m-2)}-C^\prime)<0
\end{equation*}
if $S$ is chosen large enough such that there exists a point $x\in\,\mathrm{supp}\,\rho_S$ with $|x|>1$ large enough.

We have to show that there exists a constant $K<0$ such that
\begin{equation*}
2E[\rho_R]-\int_{\mathbb{R}^N}\varepsilon\Big(\frac{2}{m}-1\Big)\rho_R^m(y)\,dy\leq K<0
\end{equation*}
for all $R\geq S$. We can use this constant $K$ as in the proof of Lemma \ref{condforcompsupp} to obtain a uniform upper bound for the diameter of the support of $\rho_R$.

Choosing a point $x\in\,\mathrm{supp}\,\rho_S$ such that $\varepsilon C^{m-2}|x|^{-N(m-2)}-C^\prime<0$, we are done if it holds that $\rho_R(x)\geq c>0$ for all $R\geq S$ since we can choose $K=c(\varepsilon C^{m-2}|x|^{-N(m-2)}-C^\prime)$. Let us assume the contrary, i.e. $\rho_R(x)\rightarrow0$ as $R\rightarrow\infty$, and note that $\rho_R$ is radially symmetric and monotonically decreasing. If $\int_{\mathbb{R}^N}\rho_\infty(x)\,dx=1$, then we have $E[\rho_\infty]=E[\rho_S]$ and $\rho_\infty\in L^m(\mathbb{R}^N)\cap\mathcal{P}(\mathbb{R}^N)$ which means that there exists a global minimizer $\rho_R$ of $E$ in $L^m(\mathbb{R}^N)\cap\mathcal{P}_R(\mathbb{R}^N)$ with $\mathrm{supp}\,\rho_R\subset\mathrm{supp}\,\rho_S$ for all $R\geq S$. We exclude $\rho_\infty\equiv0$ since $E[\rho_R]\leq E[\rho_S]<0$. Furthermore, for every $\delta>0$ we find a constant $R^\prime>0$ such that $\inf_{\rho\in L^m\cap\mathcal{P}}E[\rho]>E[\rho_R]-\delta$ for all $R\geq R^\prime$. Therefore, assuming that $\int_{\mathbb{R}^N}\rho_\infty(x)\,dx=M\in(0,1)$ we obtain a contradiction by inequality \eqref{massminrelation} (which holds for $m>2$) since $E[\rho_R]\rightarrow E[\rho_\infty]\geq\inf_{\rho\in L^m\cap\mathcal{P}^M}E[\rho]$.
\begin{lem}
\label{noexistenceofglobminformsmaller2}
Let $G$ satisfy \ref{(G1)}-\ref{(G4)}. Let $1<m\leq2$ and $\varepsilon>\frac{m}{2}\|G\|_{L^{\frac{1}{m-1}}}$, then there exists no global minimizer of $E$ in $L^m(\mathbb{R}^N)\cap\mathcal{P}(\mathbb{R}^N)$.
\end{lem}
\begin{proof}
By H\"older's inequality and Young's inequality for convolutions, we obtain for $1\leq r\leq2$
\begin{equation*}
\begin{aligned}
\|(G\ast\rho)\rho\|_{L^1} & \leq\|G\ast\rho\|_{L^{\frac{r}{r-1}}}\|\rho\|_{L^r}\leq\|G\|_{L^{\frac{r}{2r-2}}}\|\rho\|_{L^r}^2\\
 & \leq\|G\|_{L^{\frac{r}{2r-2}}}\|\rho\|_{L^{\frac{3r-2}{r}}}^{\frac{3r-2}{r}}\|\rho\|_{L^1}^{2-\frac{3r-2}{r}}.
 \end{aligned}
\end{equation*}
For $1<m\leq2$, we can choose $r\geq1$ such that $m=\frac{3r-2}{r}$ to conclude that
\begin{equation*}
E[\rho]\geq\frac{\varepsilon}{m}\|\rho\|_{L^m}^m-\frac{1}{2}\|G\|_{L^{\frac{1}{m-1}}}\|\rho\|_{L^m}^m\geq\Big(\frac{\varepsilon}{m}-\frac{1}{2}\|G\|_{L^{\frac{1}{m-1}}}\Big)\|\rho\|_{L^m}^m.
\end{equation*}
Obviously, there cannot exist a global minimizer with positive energy such that we exclude the existence of a global minimizer of $E$ in $L^m(\mathbb{R}^N)\cap\mathcal{P}(\mathbb{R}^N)$ for $\varepsilon>\frac{m}{2}\|G\|_{L^{\frac{1}{m-1}}}$.
\end{proof}
\begin{rmk}
\label{noexistenceofcompstatsolformsmaller2}
Analogously as in Lemma \ref{noexistenceofglobminformsmaller2}, we can state a result for the non-existence of compact stationary solutions of \eqref{eq1b} in case of $1<m\leq2$. Considering Lemma \ref{condforcompsupp}, we conclude that there cannot exist a compact stationary solution $\rho\in L^m(\mathbb{R}^N)\cap\mathcal{P}(\mathbb{R}^N)$ of \eqref{eq1b} for $1<m\leq2$ if $\varepsilon>\|G\|_{L^{\frac{1}{m-1}}}$.
\end{rmk}

However, these estimates do not imply that a global minimizer of $E$ in $L^m(\mathbb{R}^N)\cap\mathcal{P}(\mathbb{R}^N)$ exists for all coefficients $\varepsilon\leq\frac{m}{2}\|G\|_{L^{\frac{1}{m-1}}}$ or that a compact stationary solution $\rho\in L^m(\mathbb{R}^N)\cap\mathcal{P}(\mathbb{R}^N)$ of \eqref{eq1b} exists for all coefficients $\varepsilon\leq\|G\|_{L^{\frac{1}{m-1}}}$.

The threshold for the existence of a global minimizer of $E$ in $L^m(\mathbb{R}^N)\cap\mathcal{P}(\mathbb{R}^N)$ can be specified for $1<m\leq2$ by
\begin{equation}
\varepsilon_0\coloneqq\sup_{\rho\in L^m\cap\mathcal{P}}\frac{m}{2}\frac{\int_{\mathbb{R}^N}(G\ast\rho)(x)\rho(x)\,dx}{\|\rho\|_{L^m}^m}.
\end{equation}
\begin{lem}
Let $G$ be bounded and satisfy \ref{(G1)}-\ref{(G4)}. Let $1<m<2$, then there exists a compactly supported function $\rho\in L^m(\mathbb{R}^N)\cap\mathcal{P}(\mathbb{R}^N)$ such that $\varepsilon_0$ is attained.
\end{lem}
\begin{proof}
Let us denote by $(\rho_n)_{n\in\mathbb{N}}$ a maximizing sequence in $L^m(\mathbb{R}^N)\cap\mathcal{P}_R(\mathbb{R}^N)$ for $\frac{m}{2}\|\rho\|_{L^m}^{-m}\int_{\mathbb{R}^N}(G\ast\rho)(x)\rho(x)\,dx$. Due to $G$ being bounded and $(\rho_n)_{n\in\mathbb{N}}$ being a maximizing sequence, it follows that $(\rho_n)_{n\in\mathbb{N}}$ is bounded in $L^m(\mathbb{R}^N)$. Therefore, as in the proof of Lemma \ref{globminrestricteddomain}, we can conclude that
\begin{equation}
\varepsilon_{0,R}\coloneqq\sup_{\rho\in L^m\cap\mathcal{P}_R}\frac{m}{2}\frac{\int_{\mathbb{R}^N}(G\ast\rho)(x)\rho(x)\,dx}{\|\rho\|_{L^m}^m}
\end{equation}
is attained for some $\rho_{0,R}\in L^m(\mathbb{R}^N)\cap\mathcal{P}_R(\mathbb{R}^N)$.

Note that $(\rho_{0,R})_{R\geq R^\prime}$ is a maximizing sequence for the problem in $L^m(\mathbb{R}^N)\cap\mathcal{P}(\mathbb{R}^N)$, i.e. $\varepsilon_{0,R}\rightarrow\varepsilon_0$ as $R\rightarrow\infty$ and in particular $\varepsilon_{0,R}\geq\varepsilon_{0,R^\prime}$ for all $R\geq R^\prime$ where $R^\prime$ is chosen large enough. Due to $\|(G\ast\rho)\rho\|_{L^1}\leq\|G\|_{L^{\frac{m}{2m-2}}}\|\rho\|_{L^m}^2$, the $L^m$-norm of this maximizing sequence is bounded from below by a constant greater than zero, i.e. $\|\rho_{0,R}\|_{L^m}\geq\delta>0$ for all $R\geq R^\prime$. We use the lower bound of the $L^m$-norm to estimate that
\begin{equation}
2E[\rho_{0,R}]-\int_{\mathbb{R}^N}\varepsilon_{0,R}\Big(\frac{2}{m}-1\Big)\rho_{0,R}^m(y)\,dy\leq-\varepsilon_{0,R^\prime}
\Big(\frac{2}{m}-1\Big)\delta^m
\end{equation}
for all $R\geq R^\prime$. Choosing $K=-\varepsilon_{0,R^\prime}\big(\frac{2}{m}-1\big)\delta^m$ in Lemma \ref{condforcompsupp}, we obtain $\mathrm{supp}\,\rho_{0,R}\subset B_D(0)$ for all $R>0$ where $D$ depends on $R^\prime$ but is independent from $R\geq R^\prime$. Hence, $\rho_{0,R}\in L^m(\mathbb{R}^N)\cap\mathcal{P}_D(\mathbb{R}^N)$ for all $R>0$ and, similar as in the proof of Lemma \ref{globalminimizerformsmaller2}, it follows $\varepsilon_0=\varepsilon_{0,D}$.
\end{proof}
To sum up, we have now proved the statement of Theorem \ref{existenceofglobalmin} concerning the existence of a global minimizer of the energy $E$  in $L^m(\mathbb{R}^N)\cap\mathcal{P}(\mathbb{R}^N)$ depending on the exponent $m$ of the degenerate diffusion and the coefficient $\varepsilon$.

\section{Uniqueness of stationary solutions}
\label{stationarysolutions}
In the following, we prove for $m=2$ in arbitrary dimensions the uniqueness of stationary solutions of \eqref{eq1b} up to a translation. In \cite{burgerdifrancescofranek}, it is shown for $N=1$ that a stationary solution of \eqref{eq1b} is unique up to a translation. As in \cite{burgerfetecauhuang} for $N=1$, we also derive for $m\neq2$ in higher dimensions that there are coefficients $\varepsilon$ such that compact stationary solutions of \eqref{eq1b} exist.

First, we sketch the strategy in \cite[Section 4]{burgerdifrancescofranek} which is used to conclude for $N=1$ that up to a translation a unique stationary solution in $L^2(\mathbb{R})\cap\mathcal{P}(\mathbb{R})$ of \eqref{eq1b} exists for $m=2$ and $0<\varepsilon<1$.
\begin{itemize}
\item Stationary solutions of \eqref{eq1b} have connected support for $N=1$.
\item Stationary solutions of \eqref{eq1b} are compactly supported for $N=1$.
\item For any stationary solution of \eqref{eq1b} there exists a symmetric stationary solution of \eqref{eq1b} with the same energy if $N=1$.
\item A stationary solution $\rho\in L^2(\mathbb{R})\cap\mathcal{P}(\mathbb{R})$ of \eqref{eq1b} with $|\mathrm{supp}\,\rho|\geq|\mathrm{supp}\,\rho_{\mathrm{min}}|$, where $\rho_{\mathrm{min}}$ is a global minimizer of $E$ in $L^2(\mathbb{R})\cap\mathcal{P}(\mathbb{R})$, is a minimizer.
\item For every $L>0$, a unique function $\rho$ with $\mathrm{supp}\,\rho=[-L,L]$ exists which is symmetric, monotonically decreasing on $\{x>0\}$, compact and satisfies 
\begin{equation*}
\varepsilon\rho(x)=\int_{-L}^{L}G(x-y)\rho(y)\,dy-(G\ast\rho)(L)
\end{equation*}
for some $\varepsilon>0$. This follows considering the eigenvalue problem
\begin{equation*}
\varepsilon\rho^{\prime}(x)=\int_0^L\big(G(x-y)-G(x+y)\big)\rho^{\prime}(y)\,dy
\end{equation*}
via the strong version of the Krein-Rutman theorem (see Theorem \ref{strongversionkreinrutman}).
\item The eigenvalue $\varepsilon$ in the above eigenvalue problem is strictly monotonically increasing on $\{x>0\}$ with the size $L$ of the support of $\rho$. It holds that $\varepsilon\searrow 0$ as $L\searrow 0$ and $\varepsilon\nearrow1$ as $L\nearrow +\infty$.
\item Using the above results and that the global minimizer is symmetric and monotonically decreasing, one can show for $\varepsilon<1$ that up to a translation a unique stationary solution $\rho\in L^2(\mathbb{R})\cap\mathcal{P}(\mathbb{R})$ of \eqref{eq1b} exists for $N=1$ which coincides with the global minimizer of $E$ in $L^2(\mathbb{R})\cap\mathcal{P}(\mathbb{R})$.
\end{itemize}

Due to Theorem \ref{symmetryofstationarysolutions}, we already know that a stationary solution of \eqref{eq1b} is radially symmetric and monotonically decreasing in higher dimensions. In particular, it has a connected support. As observed in Remark \ref{mequals2nononcompact}, for $m=2$ there cannot exist a stationary solution of \eqref{eq1b} with non-compact support. By Theorem \ref{globmincomp} and Theorem \ref{existenceofglobalmin}, we also know that the global minimizer of $E$ in $L^m(\mathbb{R}^N)\cap\mathcal{P}(\mathbb{R}^N)$ is radially symmetric, monotonically decreasing and has compact support.

Following the last three points used for the proof for $N=1$, in higher dimensions we show for $m=2$ and $0<\varepsilon<1$ that there is a unique stationary solution of \eqref{eq1b} up to a translation which coincides with the global minimizer of the energy $E$ in $L^m(\mathbb{R}^N)\cap\mathcal{P}(\mathbb{R}^N)$. Moreover, as in \cite[Section 3.1]{burgerfetecauhuang} for $N=1$, we deduce in higher dimensions in case of $m\neq 2$ for any $R>0$ the existence of a coefficient $\varepsilon>0$ such that there is a radially symmetric stationary solution of \eqref{eq1b} with support $\overline{B_R(0)}$.

For radially symmetric functions $\rho$ and $G$, i.e.\ $\rho(x)=\tilde{\rho}(|x|)$ and $G(x)=g(|x|)$, we can write
\begin{equation*}
\begin{aligned}
(G\ast\rho)(x) & =\int_{\mathbb{R}^N}G(x-y)\rho(y)\,dy=\int_0^{+\infty}\Big(\int_{\partial B_s(0)}G(x-y)\,d\sigma(y)\Big)\tilde{\rho}(s)\,ds\\
& =\int_0^{+\infty}\Big(\int_{\partial B_s(0)}G(|x|e_1-y)\,d\sigma(y)\Big)\tilde{\rho}(s)\,ds.
\end{aligned}
\end{equation*}
Using that $\int_{\partial B_s(0)}\nabla G(x-y)\,d\sigma(y)$ is parallel to $x$ and rotationally invariant, it is shown in \cite[Lemma 3.2]{bertozzicarrillolaurent} that it holds
\begin{equation}
\nabla(G\ast\rho)(x)=\int_0^{+\infty}\Big(\int_{\partial B_s(0)}\nabla G(|x|e_1-y)\cdot e_1\,d\sigma(y)\Big)\tilde{\rho}(s)\,ds\,\frac{x}{|x|}.
\end{equation}
Similarly, we also derive
\begin{equation}
\nabla(G\ast\rho)(x)=\int_0^{+\infty}\Big(\int_{\partial B_s(0)}G(|x|e_1-y)\frac{y\cdot e_1}{|y|}\,d\sigma(y)\Big)\tilde{\rho}^{\prime}(s)\,ds\,\frac{x}{|x|}.
\end{equation}
Using the radial symmetry of $G$, it is useful to observe that we have
\begin{equation}
\label{rotationjustwithg}
\int_{\partial B_s(0)}G(re_1-y)\,d\sigma(y)=\frac{s^{N-1}}{r^{N-1}}\int_{\partial B_r(0)}G(se_1-y)\,d\sigma(y)
\end{equation}
and
\begin{equation}
\label{rotationwithgandcos}
\int_{\partial B_s(0)}G(re_1-y)\frac{y\cdot e_1}{|y|}\,d\sigma(y)=\frac{s^{N-1}}{r^{N-1}}\int_{\partial B_r(0)}G(se_1-y)\frac{y\cdot e_1}{|y|}\,d\sigma(y).
\end{equation}
In the radial case, the energy functional reads
\begin{equation*}
\begin{aligned}
E[\tilde{\rho}]= & \int_0^{+\infty}\frac{\varepsilon}{m}\tilde{\rho}^m(r)\omega_N r^{N-1}\,dr\\
& -\frac{1}{2}\int_0^{+\infty}\int_0^{+\infty}\omega_N r^{N-1}\tilde{\rho}(r)\Big(\int_{\partial B_s(0)}G(re_1-y)\,d\sigma(y)\Big)\tilde{\rho}(s)\,dsdr
\end{aligned}
\end{equation*}
where $\omega_N$ denotes the surface area of the unit sphere in $\mathbb{R}^N$.

Let us now state the strong version of the Krein-Rutman theorem as in \cite[Theorem 4.10]{burgerdifrancescofranek}.
\begin{thm}[Krein-Rutman theorem, strong version]
\label{strongversionkreinrutman}
Let $X$ be a Banach space. Let $K\subset X$ be a solid cone, i.e.\ $\lambda K\subset K$ for all $\lambda\geq 0$ and $K$ has a non-empty interior $K_0$. Let $T$ be a compact linear operator which is strongly positive with respect to $K$, i.e.\ $T[u]\in K_0$ if $u\in K\setminus\{0\}$. Then:
\begin{itemize}
\item[(i)] The spectral radius $r(T)$ is strictly positive and $r(T)$ is a simple eigenvalue with an eigenvector $v\in K_0$. There is no other eigenvalue with a corresponding eigenvector $v\in K$.
\item[(ii)] $|\lambda|<r(T)$ for all other eigenvalues $\lambda\neq r(T)$.
\end{itemize}
\end{thm}
Following the approach for $N=1$ in \cite[Section 4]{burgerdifrancescofranek}, finding in higher dimensions a compact, radially symmetric and monotonically decreasing function $\rho\in C^2(\mathrm{supp}\,\rho)$ vanishing on the boundary and satisfying 
\begin{equation*}
\varepsilon\rho(x)=(G\ast\rho)(x)-(G\ast\rho)(\tilde{x})
\end{equation*}
with $\tilde{x}\in\partial(\mathrm{supp}\,\rho)$, means finding $\tilde{\rho}\in C^2(\mathrm{supp}\,\tilde{\rho})$ with $\mathrm{supp}\,\tilde{\rho}=[0,R]$ such that
\begin{equation*}
\begin{aligned}
& \tilde{\rho}(R)=0,\quad -\tilde{\rho}^{\prime}(r)=u(r),\quad u\geq 0, \\
& \varepsilon u=\int_0^R H(r,s)u(s)\,ds\quad\mathrm{with}\quad H(r,s)\coloneqq\int_{\partial B_s(0)}G(re_1-y)\frac{y\cdot e_1}{|y|}\,d\sigma(y).
\end{aligned}
\end{equation*}
This is equivalent since we have
\begin{equation*}
\begin{aligned}
\varepsilon\rho(x) & =\varepsilon\int_{|x|}^R -\tilde{\rho}^{\prime}(r)\,dr\\
& =-\int_{|x|}^R\int_0^R\Big(\int_{\partial B_s(0)}G(re_1-y)\frac{y\cdot e_1}{|y|}\,d\sigma(y)\Big)\tilde{\rho}^{\prime}(s)\,dsdr\\
& =-\int_{|x|}^R\int_0^R\Big(\int_{\partial B_s(0)}\partial_r G(re_1-y)\,d\sigma(y)\Big)\tilde{\rho}(s)\,dsdr\\
& =-\int_0^R\Big(\int_{\partial B_s(0)}G(Re_1-y)\,d\sigma(y)\Big)\tilde{\rho}(s)\,ds\\
& \quad +\int_0^R\Big(\int_{\partial B_s(0)}G(|x|e_1-y)\,d\sigma(y)\Big)\tilde{\rho}(s)\,ds\\
& =(G\ast\rho)(x)-(G\ast\rho)(\tilde{x})
\end{aligned}
\end{equation*}
with $|\tilde{x}|=R$.

In order to simplify notation, let us define
\begin{align}
Y_R & \coloneqq\{\tilde{\rho}\in C([0,R])\,\big|\,\tilde{\rho}(R)=0\},\\
\mathcal{H}_R[u](r) & \coloneqq \int_0^R\Big(\int_{\partial B_s(0)}G(re_1-y)\frac{y\cdot e_1}{|y|}\,d\sigma(y)\Big)u(s)\,ds,\\
\mathcal{G}_R[\tilde{\rho}] & \coloneqq \int_0^R\Big(\int_{\partial B_s(0)}G(re_1-y)-G(Re_1-y)\,d\sigma(y)\Big)\tilde{\rho}(s)\,ds.
\end{align}

To prove the following result about uniqueness of a function with the above properties via the strong version of the Krein-Rutman theorem as in \cite[Proposition 4.11]{burgerdifrancescofranek} for $N=1$, it is enough to show that if $u\geq 0$, then $\mathcal{H}_R[u](r)\geq 0$ holds for all $r\in[0,R]$ and $\mathcal{H}_R[u](0)=0$. Moreover, for $u\in\{f\in C^1([0,R])\,\big|\,f(0)=0\}$ satisfying $u\geq0$ and $u\not\equiv0$ we must have $(\mathcal{H}[u])^{\prime}(0)>0$.
\begin{thm}
\label{uniquenessviakreinrut}
Let $G$ satisfy \ref{(G1)}-\ref{(G4)}. For every $R>0$ there exists a unique, radially symmetric function $\rho\in C^2(\overline{B_R(0)})\cap\mathcal{P}(\mathbb{R}^N)\cap C(\mathbb{R}^N)$ with $\mathrm{supp}\,\rho=\overline{B_R(0)}$ and with radial representative $\tilde{\rho}$ such that $\tilde{\rho}^{\prime}(r)\leq 0$ for $r\geq0$, $\tilde{\rho}^{\prime\prime}(0)<0$ and such that $\rho$ solves 
\begin{equation*}
\varepsilon\rho(x)=(G\ast\rho)(x)-(G\ast\rho)(R)
\end{equation*}
in $\overline{B_R(0)}$ for some coefficient $\varepsilon=\varepsilon(R)>0$.

Moreover, $\varepsilon(R)$ is the largest eigenvalue of the compact operator $\mathcal{G}_R$ in the Banach space $Y_R$ and any other eigenfunction of $\mathcal{G}_R$ in $Y_R$ with
\begin{equation*}
\int_0^{+\infty}\tilde{\rho}(r)\omega_N r^{N-1}\,dr=1
\end{equation*}
has the corresponding eigenvalue $\varepsilon^{\prime}$ satisfying $|\varepsilon^{\prime}|<\varepsilon(R)$.
\end{thm}
\begin{proof}
Let us split up $H(r,s)$ into
\begin{equation*}
\begin{aligned}
H(r,s) & =\int_{\partial B_s(0)\cap\{y_1>0\}}G(re_1-y)\frac{y\cdot e_1}{|y|}\,d\sigma(y)\\
& \quad +\int_{\partial B_s(0)\cap\{y_1<0\}}G(re_1-y)\frac{y\cdot e_1}{|y|}\,d\sigma(y).
\end{aligned}
\end{equation*}
For any $y_1>0$ there exists $\tilde{y}_1<0$ such that we have $-y_1=\tilde{y}_1$ as well as $\frac{y\cdot e_1}{|y|}=-\frac{\tilde{y}\cdot e_1}{|\tilde{y}|}$ with $y=(y_1,\ldots,y_N)$ and $\tilde{y}=(\tilde{y}_1,\ldots,y_N)$ on $\partial B_s(0)$.

Let $r\geq 0$, then it holds that $|r-y_1|\leq|r-\tilde{y}_1|$, i.e.\ we have $|re_1-y|\leq|re_1-\tilde{y}|$ and $g(|re_1-y|)\geq g(|re_1-\tilde{y}|)$ and equality holds if and only if $r=0$ since $g$ is strictly monotonically decreasing.

Therefore, we have $G(re_1-y)\frac{y\cdot e_1}{|y|}\geq-G(re_1-\tilde{y})\frac{\tilde{y}\cdot e_1}{|\tilde{y}|}$ such that $H(r,s)\geq 0$ for all $r\in[0,R]$ and an equality holds if and only if $r=0$. Due to $u$ being non-negative, we directly see that $\mathcal{H}_R[u](r)\geq 0$ holds for all $r\in[0,R]$ and $\mathcal{H}_R[u](0)=0$.

Finally, calculating the derivative of $\mathcal{H}_R[u]$ at the point $r=0$ for $u\not\equiv 0$
\begin{equation*}
\frac{d}{dr}\mathcal{H}_R[u](r)\Big|_{r=0}=-\int_0^{+\infty}\int_{\partial B(0,s)}g^{\prime}(|y|)\frac{y_1^2}{|y|^2}\,d\sigma(y)\,u(s)\,ds>0
\end{equation*}
yields the desired inequality.
\end{proof}
Using the properties of the radially symmetric case as in Theorem \ref{uniquenessviakreinrut}, following the proofs for $N=1$ in \cite[Section 3.1]{burgerfetecauhuang} we also obtain in case of $m\neq2$ for all $R>0$ the existence of a compact, radially symmetric stationary solution of \eqref{eq1b} with $\mathrm{supp}\,\rho=\overline{B_R(0)}$ in higher dimensions. Due to the additional non-linearity arising for $m\neq2$, it seems to be difficult to obtain the uniqueness as in Theorem \ref{uniquenessviakreinrut}. 
\begin{thm}
\label{existenceviakreinrutmanforgeneralm}
Let $G$ satisfy \ref{(G1)}-\ref{(G4)}. For every $R>0$ there exists a radially symmetric function $\rho\in\mathcal{P}(\mathbb{R}^N)\cap C(\mathbb{R}^N)$ with $\mathrm{supp}\,\rho=\overline{B_R(0)}$ and with radial representative $\tilde{\rho}$ such that $\tilde{\rho}^{\prime}(r)\leq0$ for $r\geq0$, $\tilde{\rho}(R)=0$ and $\rho$ solves
\begin{equation*}
\varepsilon\rho^{m-1}(x)=(G\ast\rho)(x)-(G\ast\rho)(R)
\end{equation*}
in $\overline{B_R(0)}$ for some coefficient $\varepsilon>0$.
\end{thm}
Analogously to the one-dimensional case in \cite[Proposition 4.12]{burgerdifrancescofranek}, we can prove the following result in higher dimensions.
\begin{lem}
\label{eigenvalueandsupport}
In case of $m=2$ the simple eigenvalue $\varepsilon(R)$ from Theorem \ref{uniquenessviakreinrut} is uniquely determined as a function of $R$ and $\varepsilon(R)$ is strictly monotonically increasing with $R$, $\varepsilon(R)\searrow 0$ as $R\searrow 0$ and $\varepsilon(R)\nearrow1$ as $R\nearrow +\infty$.
\end{lem}
\begin{proof}
We adapt the approach in \cite[Proposition 4.12]{burgerdifrancescofranek} for $N=1$ to our setting. Let $u_R$ denote the unique eigenfunction from Theorem \ref{uniquenessviakreinrut} with corresponding eigenvalue $\varepsilon(R)$. We conclude, as was the idea in the proof of \cite[Proposition 4.12]{burgerdifrancescofranek}, by considering
\begin{equation}
\varepsilon(R)u_R(r)=\mathcal{H}_R[u_R](r),
\end{equation}
multiplying by $r^{N-1}u_{R+\delta}(r)$ and taking into consideration that $H(r,s)=\frac{s^{N-1}}{r^{N-1}}H(s,r)$ due to \eqref{rotationwithgandcos}. Then, we obtain for every $\delta>0$
\begin{equation*}
\begin{aligned}
& \varepsilon(R)\int_0^R r^{N-1}u_R(r)u_{R+\delta}(r)\,dr=\int_0^R r^{N-1}\mathcal{H}_R[u_R](r)u_{R+\delta}(r)\,dr\\
& \quad=\int_0^R\int_0^R H(r,s)u_R(s)\,ds\,u_{R+\delta}(r)\,dr\\
& \quad=\int_0^R s^{N-1}u_R(s)\Big(\int_0^{R+\delta}H(s,r)u_{R+\delta}(r)\,dr\Big)ds\\
& \quad\quad-\int_0^R s^{N-1}u_R(s)\Big(\int_R^{R+\delta}H(s,r)u_{R+\delta}(r)\,dr\Big)ds\\
& \quad=\varepsilon(R+\delta)\int_0^R s^{N-1}u_R(s)u_{R+\delta}(s)\,ds\\
& \quad\quad-\int_0^R s^{N-1}u_R(s)\Big(\int_R^{R+\delta}H(s,r)u_{R+\delta}(r)\,dr\Big)ds.
\end{aligned}
\end{equation*}
Since we know by the proof of Theorem \ref{uniquenessviakreinrut} that it holds $u_R(r)>0$ for $r\in(0,R]$, we have shown that $\varepsilon(R+\delta)>\varepsilon(R)$.
\end{proof}
Now, we prove for $m=2$ in arbitrary dimensions that there is a unique stationary solution of \eqref{eq1b} up to a translation.
\begin{proof}[Proof of Theorem \ref{uniquenesstheoremformequals2}]
By Theorem \ref{existenceofglobalmin}, we know that a radially symmetric and monotonically decreasing global minimizer of $E$ in $L^2(\mathbb{R}^N)\cap\mathcal{P}(\mathbb{R}^N)$ exists for $\varepsilon<1$. Therefore, the minimizer is connected and has compact support as shown in Theorem \ref{globmincomp}.
Due to Corollary \ref{propcompandnoncompsol}, Theorem \ref{uniquenessviakreinrut} and Lemma \ref{eigenvalueandsupport}, we know that a unique stationary solution of \eqref{eq1b} with these properties exists such that we conclude that the global minimizer of the energy $E$ in $L^2(\mathbb{R}^N)\cap\mathcal{P}(\mathbb{R}^N)$ is unique.

Moreover, by Corollary \ref{propcompandnoncompsol} and Remark \ref{mequals2nononcompact}, for $m=2$ no stationary solution $\rho\in L^2(\mathbb{R}^N)\cap\mathcal{P}(\mathbb{R}^N)$ of \eqref{eq1b} with a non-compact support and due to Theorem \ref{symmetryofstationarysolutions} no stationary solution of \eqref{eq1b} not being radially symmetric and monotonically decreasing exists. Therefore, we have proved that a stationary solution $\rho\in L^m(\mathbb{R}^N)\cap\mathcal{P}(\mathbb{R}^N)$ of \eqref{eq1b} is unique.
\end{proof}

\section{Discussion of stationary solutions with positive energy}
\label{discussion}
In this section, we discuss if a stationary solution of \eqref{eq1b} with positive energy may exist. We assume that the interaction potential $G$ satisfies \ref{(G1)}-\ref{(G4)}.

We know by Theorem \ref{uniquenesstheoremformequals2} that for $m=2$ there exists for every $0<\varepsilon<1$ and $R>0$ a triple $(\varepsilon,R,\rho)$ where each component uniquely determines the other ones such that
\begin{equation*}
\rho\nabla(\varepsilon\rho^{m-1}-G\ast\rho)=0
\end{equation*}
is satisfied. Here, $\rho$ is a radially symmetric and monotonically decreasing function $\rho\in C^2(\overline{B_R(0)})\cap\mathcal{P}(\mathbb{R}^N)\cap C(\mathbb{R}^N)$ with $\mathrm{supp}\,\rho=\overline{B_R(0)}$.

If one would like to derive this result for $m\neq2$, even when being able to prove uniqueness in Theorem \ref{existenceviakreinrutmanforgeneralm}, which is indicated by numerical calculations in \cite[Section 4.2]{burgerfetecauhuang}, one cannot simply follow the approach in the proof of Lemma \ref{eigenvalueandsupport} since for general $m>1$ one obtains
\begin{equation*}
\begin{aligned}
& \varepsilon(R+\delta)(m-1)\int_0^R r^{N-1}\tilde{\rho}_{R+\delta}^{m-2}(r)u_R(r)u_{R+\delta}(r)\,dr\\
& -\varepsilon(R)(m-1)\int_0^R r^{N-1}\tilde{\rho}_R^{m-2}(r)u_R(r)u_{R+\delta}(r)\,dr\\
& \quad=\int_0^R r^{N-1}u_R(r)\Big(\int_R^{R+\delta}H(r,s)u_{R+\delta}(s)\,ds\Big)dr.
\end{aligned}
\end{equation*}

Out of this equation, one could still conjecture that the coefficient $\varepsilon$ is strictly monotonically increasing with the size of the support as is also suggested by numerical results in \cite[Section 4.2]{burgerfetecauhuang}.

However, in case of $m>2$ we can at least show that the size of the support of a compact stationary solution of \eqref{eq1b} is bounded from below by the coefficient $\varepsilon$. Using the reverse H\"older inequality, we obtain
\begin{equation}
\label{estimateforrhoinnorm}
\varepsilon\|\rho\|_{L^{m-1}}^{m-1}=\varepsilon\|\rho^{m-1}\|_{L^1}\geq\varepsilon\|\rho\|_{L^1}^{m-1}|\mathrm{supp}\,\rho|^{2-m}=\varepsilon|\mathrm{supp}\,\rho|^{2-m}.
\end{equation}
Since we assume the stationary solution of \eqref{eq1b} to be compactly supported, it holds that $\varepsilon\rho^{m-1}(x)<G\ast\rho(x)$ in $\mathrm{supp}\,\rho$. Integrating over $\mathrm{supp}\,\rho$ and extending the integration domain on the right hand side to the whole space, using inequality \eqref{estimateforrhoinnorm} we obtain that
\begin{equation*}
\varepsilon|\mathrm{supp}\,\rho|^{2-m}<1.
\end{equation*}
So, we derived for $m>2$ a lower bound for the size of the support depending on the coefficient $\varepsilon$ which reads for a radially symmetric and monotonically decreasing stationary solution of \eqref{eq1b} with $\mathrm{supp}\,\rho=\overline{B_R(0)}$ as
\begin{equation*}
|B_R(0)|>\varepsilon^{\frac{1}{m-2}}.
\end{equation*}
Such a relation is also obtained for $N=1$ in \cite[Section 4.2]{burgerfetecauhuang} by assuming the stationary solution of \eqref{eq1b} to be approximately a characteristic function which is suitably scaled to have unit mass.

If $m>2$, then there exists for every $\varepsilon>0$ a radially symmetric and monotonically decreasing stationary solution $\rho\in L^m(\mathbb{R}^N)\cap\mathcal{P}(\mathbb{R}^N)$ of \eqref{eq1b} being compactly supported due to Theorem \ref{existenceofglobalmin}, the compactness of global minimizers of $E$ in $L^m(\mathbb{R}^N)\cap\mathcal{P}(\mathbb{R}^N)$ (Theorem \ref{globmincomp}) and Lemma \ref{resultsfromburgerdifrancescofranek}\ref{results(iv)}.

In contrast, if $1<m<2$, we only know that there exists for all coefficients $\varepsilon$ not greater than some $\varepsilon_0>0$ a radially symmetric and monotonically decreasing stationary solution $\rho\in L^m(\mathbb{R}^N)\cap\mathcal{P}(\mathbb{R}^N)$ of \eqref{eq1b} being compactly supported due to Theorem \ref{existenceofglobalmin}, Theorem \ref{globmincomp} and Lemma \ref{resultsfromburgerdifrancescofranek}\ref{results(iv)}. The constant $\varepsilon_0$ depends on the interaction potential $G$ and the exponent $m$ of the degenerate diffusion (cf.\ Remark \ref{noexistenceofcompstatsolformsmaller2}) and marks the threshold where we can find a function $\rho\in L^m(\mathbb{R}^N)\cap\mathcal{P}(\mathbb{R}^N)$ such that we have $E[\rho]\leq0$.

These observations complement the results in the theoretical part of \cite{burgerfetecauhuang}. It is shown in \cite[Theorem 3.7]{burgerfetecauhuang} for $N=1$ that for $m>2$ there exists for every $L>0$ a coefficient $\varepsilon>0$ such that there is a stationary solution of \eqref{eq1b} with $\mathrm{supp}\,\rho=[-L,L]$ which is symmetric and monotonically decreasing on $\{x>0\}$. We extended this result in Theorem \ref{existenceviakreinrutmanforgeneralm} to arbitrary dimensions. In the numerical part of \cite{burgerfetecauhuang}, it is suggested that there may exist a compactly supported stationary solution of \eqref{eq1b} for all coefficients $\varepsilon>0$ which is proved here as remarked above.

Moreover, for $1<m<2$ it is shown in \cite[Theorem 3.9]{burgerfetecauhuang} for $N=1$ that for every $L>0$ there exists a coefficient $\varepsilon>0$ such that there is a stationary solution of \eqref{eq1b} with $\mathrm{supp}\,\rho=[-L,L]$ which is symmetric and monotonically decreasing on $\{x>0\}$. Again, we extended this result in Theorem \ref{existenceviakreinrutmanforgeneralm} to arbitrary dimensions. In the numerical part of \cite{burgerfetecauhuang}, it is suggested that in this case a compactly supported stationary solution of \eqref{eq1b} exists for any coefficient $\varepsilon$ smaller than a constant $\varepsilon_1>0$ and that these stationary solutions are just local minimizers but could turn into global minimizers for coefficients $\varepsilon$ smaller than some $\varepsilon_0<\varepsilon_1$. As remarked above, this statement is proved here only in terms of stationary solutions of \eqref{eq1b} which are global minimizers of $E$ in $L^m(\mathbb{R}^N)\cap\mathcal{P}(\mathbb{R}^N)$, i.e.\ for $0<\varepsilon\leq\varepsilon_0$.

Obviously, and also pointed out in Lemma \ref{noexistenceofglobminformsmaller2} and Remark \ref{noexistenceofcompstatsolformsmaller2}, for $1<m<2$ the condition for being a compact stationary solution of \eqref{eq1b} is less strict than for being a global minimizer of $E$ in $L^m(\mathbb{R}^N)\cap\mathcal{P}(\mathbb{R}^N)$ because the latter forces the energy to be non-positive which does not need to hold for compactness of stationary solutions of \eqref{eq1b} in case of $1<m<2$ (cf.\ Remark \ref{stationarysolutionswithpositiveenergy}). 

By Corollary \ref{propcompandnoncompsol}, we know that we can write the energy of a compactly supported stationary solution $\rho\in L^m(\mathbb{R}^N)\cap\mathcal{P}(\mathbb{R}^N)$ of \eqref{eq1b} with $\mathrm{supp}\,\rho=\overline{B_R(0)}$ as
\begin{equation*}
E[\rho]=\int_{\mathbb{R}^N}\varepsilon\Big(\frac{1}{m}-\frac{1}{2}\Big)\rho^m(y)\,dy-\frac{1}{2}(G\ast\rho)(R).
\end{equation*}
So, depending on the interaction potential $G$ and the size of the support of the stationary solution of \eqref{eq1b} there could exist a compactly supported stationary solution of \eqref{eq1b} with positive energy for $1<m<2$.

Now, assume that there is a triple $(\varepsilon,R,\rho)$ where each component uniquely determines the other ones and which solves $\rho\nabla(\varepsilon\rho^{m-1}-G\ast\rho)=0$ with $\rho$ being a radially symmetric and monotonically decreasing function $\rho\in C^2(B_R(0))\cap\mathcal{P}(\mathbb{R}^N)\cap C(\mathbb{R}^N)$ and $\mathrm{supp}\,\rho=\overline{B_R(0)}$. In addition, assume that $\varepsilon$ is strictly increasing with the size of the support. Then, there has to exist a compact stationary solution $\rho\in L^m(\mathbb{R}^N)\cap\mathcal{P}(\mathbb{R}^N)$ of \eqref{eq1b} which is no global minimizer. To convince ourselves about that let $\varepsilon=\varepsilon_0$, i.e.\ there exists a $\rho_0\in L^m(\mathbb{R}^N)\cap\mathcal{P}(\mathbb{R}^N)$ such that $E[\rho_0]=0$ and there is no $\rho\in L^m(\mathbb{R}^N)\cap\mathcal{P}(\mathbb{R}^N)$ with $E[\rho]<0$. Then, by Theorem \ref{existenceofglobalmin} we know that $\rho_0$ is a global minimizer of $E$ which has compact support due to Theorem \ref{globmincomp}. Considering Theorem \ref{existenceviakreinrutmanforgeneralm}, there is a radially symmetric and monotonically decreasing stationary solution $\rho\in\mathcal{P}(\mathbb{R}^N)\cap C(\mathbb{R}^N)$ of \eqref{eq1b} with $\mathrm{supp}\,\rho\supset\,\mathrm{supp}\,\rho_0$. Since we assumed the coefficient $\varepsilon$ to increase strictly with $R$, we have $E[\rho]>0$ because of $\varepsilon>\varepsilon_0$.

To sum up, if we were able to prove uniqueness in Theorem \ref{existenceviakreinrutmanforgeneralm} for $1<m<2$ and if we were able to prove that the coefficient $\varepsilon$ is strictly increasing with the size of the support of the resulting unique function, then we would have shown that a unique compactly supported, radially symmetric and monotonically decreasing stationary solution of \eqref{eq1b} exists for coefficients $\varepsilon$ smaller than some value $\varepsilon_1$. This stationary solution is a global minimizer of $E$ in $L^m(\mathbb{R}^N)\cap\mathcal{P}(\mathbb{R}^N)$ for coefficients $\varepsilon\leq\varepsilon_0$ but loses this property for coefficients with a larger value. In particular, we would have shown that the threshold $\varepsilon_0$ for the existence of a global minimizer of $E$ in $L^m(\mathbb{R}^N)\cap\mathcal{P}(\mathbb{R}^N)$ is strictly smaller than the threshold $\varepsilon_1$ for the existence of a compact stationary solution of \eqref{eq1b}. 

The value $\varepsilon_1$ is presumably determined by the condition for a stationary solution $\rho\in L^m(\mathbb{R}^N)\cap\mathcal{P}(\mathbb{R}^N)$ of \eqref{eq1b} to be compact, i.e.\ by Remark \ref{differentcompcond} and Theorem \ref{symmetryofstationarysolutions} it is presumably the smallest coefficient such that there is no radially symmetric and monotonically decreasing $\rho\in L^m(\mathbb{R}^N)\cap\mathcal{P}(\mathbb{R}^N)$ satisfying
\begin{equation*}
\varepsilon\rho^{m-1}(x)-(G\ast\rho)(x)<0
\end{equation*}
in $\mathrm{supp}\,\rho$. In \cite[Section 4.2]{burgerfetecauhuang}, it is also shown formally for $1<m<2$ that the non-linear eigenvalue problem
\begin{equation*}
\varepsilon\rho^{m-1}(x)-(G\ast\rho)(x)=0
\end{equation*}
governs for $N=1$ the limiting profile for a stationary solution of \eqref{eq1b} with support $[-L,L]$ and $L\rightarrow\infty$.

So, for $1<m<2$ the threshold for the existence of a compact stationary solution $\rho\in L^m(\mathbb{R}^N)\cap\mathcal{P}(\mathbb{R}^N)$ of \eqref{eq1b} can be estimated by
\begin{equation}
\varepsilon_1<\frac{2}{m}\varepsilon_0=\sup_{\rho\in L^m\cap\mathcal{P}}\frac{\int_{\mathbb{R}^N}(G\ast\rho)(x)\rho(x)\,dx}{\|\rho\|_{L^m}^m}.
\end{equation}
We obtain a strict inequality since the function attaining the supremum is a global minimizer of the energy $E$ in $L^m(\mathbb{R}^N)\cap\mathcal{P}(\mathbb{R}^N)$ with coefficients $\varepsilon_0$ and has to be compactly supported.

Considering equation \eqref{equationforconnectednoncompactsolution} in Corollary \ref{propcompandnoncompsol}, we know that every non-compact stationary solution $\rho\in L^m(\mathbb{R}^N)\cap\mathcal{P}(\mathbb{R}^N)$ of \eqref{eq1b} satisfies $\varepsilon=\|\rho\|_{L^m}^{-m}\int_{\mathbb{R}^N}(G\ast\rho)(x)\rho(x)\,dx$. Therefore, we can estimate that under our assumptions a stationary solution $\rho\in L^m(\mathbb{R}^N)\cap\mathcal{P}(\mathbb{R}^N)$ of \eqref{eq1b} cannot exist for $1<m<2$ if $\varepsilon\geq\frac{2}{m}\varepsilon_0$.

\section*{Acknowledgements}
The author would like to thank Angela Stevens for numerous critical discussions and financial support, and Martin Burger for some useful comments and suggestions.

\bibliography{references}

\end{document}